\documentclass[11pt]{amsart}

\usepackage{amsthm}
\usepackage{amsmath}
\usepackage{amssymb}
\usepackage[margin=1in]{geometry}
\usepackage{enumerate}
\usepackage{hyperref}
\usepackage{mathrsfs}
\usepackage{color}

\title{The Selgrade decomposition for linear semiflows on Banach spaces}
\author{Alex Blumenthal}
\address{Alex Blumenthal \\ 4305 Kirwan Hall \\ University of Maryland \\ College Park, MD 20742}
\email{alexb123@math.umd.edu}
\urladdr{http://math.umd.edu/~alexb123/}
\thanks{First author: This material is based upon work supported by the National Science Foundation under Award No. 1604805. Second author: Partially supported by the NSF grant DMS-1067929, by the Research Board and Research Council of the University of Missouri, and by the Simons Foundation.}

\author{Yuri Latushkin}
\address{Yuri Latushkin, 104 Mathematics Bldg., University of Missouri, Columbia, MO 65211 }
\email{latushkiny@missouri.edu}
\urladdr{https://faculty.missouri.edu/~latushkiny/}


\theoremstyle{theorem}
\newtheorem{thm}{Theorem}[section]
\newtheorem{thmA}{Theorem}

\newtheorem{cor}[thm]{Corollary}
\newtheorem{lem}[thm]{Lemma}
\newtheorem{prop}[thm]{Proposition}

\theoremstyle{definition}

\newtheorem{defn}[thm]{Definition}
\newtheorem{rmk}[thm]{Remark}

\newtheorem{cla}[thm]{Claim}
\newtheorem{ex}[thm]{Example}

\newcommand{\N}{\mathbb{N}}
\renewcommand{\P}{\mathbb{P}}
\newcommand{\R}{\mathbb{R}}

\newcommand{\Z}{\mathbb{Z}}

\newcommand{\Bc}{\mathcal{B}}
\newcommand{\Fc}{\mathcal{F}}

\newcommand{\Gc}{\mathcal{G}}
\newcommand{\Ac}{\mathcal{A}}

\newcommand{\Nc}{\mathcal{N}}
\renewcommand{\Mc}{\mathcal M}
\newcommand{\Ec}{\mathcal E}
\newcommand{\Vc}{\mathcal V}

\newcommand{\Cc}{\mathcal C}
\newcommand{\Zc}{\mathcal Z}

\renewcommand{\a}{\alpha}

\renewcommand{\d}{\delta}
\newcommand{\e}{\epsilon}

\newcommand{\graph}{\operatorname{graph}}

\newcommand{\dist}{\operatorname{dist}}
\newcommand{\ds}{/ \! /}
\newcommand{\codim}{\operatorname{codim}}

\newcommand{\Gap}{\operatorname{Gap}}

\begin{document}

\begin{abstract}
We extend Selgrade's Theorem, Morse spectrum, and related concepts to the setting of linear skew product semiflows on a separable Banach bundle. We recover a characterization, well-known in the finite-dimensional setting, of exponentially separated subbundles as attractor-repeller pairs for the associated semiflow on the projective bundle.
\end{abstract}

\maketitle
\centerline{\it Dedicated to the memory of George Sell, to whom we owe so much.}

\medskip

\section{Introduction and statement of results}\label{sec:introResults}

In a brilliant series of papers by George Sell \cite{sacker1974existence, sacker1976existenceII, sacker1976existenceIII, sacker1978existence, sacker1978spectral} and his collaborators and contemporaries
\cite{johnson1980analyticity, johnson1987ergodic, salamon1988flows, selgrade1975isolated}, 
a foundation of the modern theory of finite-dimensional linear skew product flows was laid out and numerous connections to ordinary differential equations were established. Moreover, there is by now a considerable literature dedicated to the treatment of partial differential equations as dynamical systems. Of particular interest for dynamicists are dissipative PDE, for example dissipative parabolic problems (e.g., Navier-Stokes in two dimensions and reaction-diffusion equations) and dispersive wave equations. Many such equations can be thought of as differentiable dynamical systems on infinite-dimensional Hilbert or Banach spaces \cite{henry1981geometric, sell2013dynamics}. Moreover, many such systems admit global compact attractors \cite{babin1992attractors, hale2010asymptotic}, and so can be studied using techniques adapted from classical dynamical systems theory for finite-dimensional systems. For more information we refer the reader to \cite{babin2006global, carvalho2013attractors, kloeden2011nonautonomous, temam2012infinite}.

\medskip

Let us restrict the discussion to a certain subclass of techniques: decompositions of the tangent bundle into continuous and measurables subbundles and associated spectra, for now in the finite-dimensional setting. Notable such decompositions include those of Sacker-Sell \cite{sacker1974existence, johnson1987ergodic, sacker1978spectral}, Selgrade \cite{selgrade1975isolated}, and Oseledets \cite{oseledets1968multiplicative}; see also \cite{grune2000uniform, rasmussen2009dichotomy, rasmussen2010alternative}, and see \cite{colonius2012dynamics} for a general reference. These objects are useful in a variety of ways, e.g., in establishing existence of invariant manifolds for the dynamics. 

Many properties of the Sacker-Sell decomposition and spectrum have been extended to a setting amenable to applications to PDE by Sacker and Sell \cite{sacker1994dichotomies} and many others \cite{chicone1999evolution, chow1995existence, chow1996two, chow1996unbounded, magalhaes1987spectrum, shen1998almost, shvydkoy2006cocycles}. To briefly review, this decomposition splits the tangent bundle into (continuous) subbundles, possibly infinite-dimensional, each pair of which satisfies \emph{exponential dichotomy}: two subbundles have an exponential dichotomy if there exists some $\lambda \in \R$ such that the smallest asymptotic exponential growth rate on one subbundle is strictly larger $\lambda$ at every base point, while the largest exponential growth rates on the other subbundle is strictly smaller.

The Oseledets decomposition and associated Lyapunov spectrum (a.k.a. Lyapunov exponents) has been similarly extended to the setting of cocycles of linear operators on infinite-dimensional Banach spaces; see, e.g., \cite{mane1983lyapounov, ruelle1982characteristic, thieullen1987fibres}, as well as \cite{Blumenthal20162377} and the literature cited therein. The Oseledets decomposition is really an aspect of the ergodic theory of a linear cocycle: roughly, it can be thought of as the 'measurable' counterpart to the Sacker-Sell decomposition. In particular, Oseledets subbundles are defined only at almost every base point with respect to a given invariant measure on the base, and vary measurably as opposed to continuously in the fiber. Consequently, the Oseledets decomposition is typically much finer than the Sacker-Sell decomposition; see, e.g., \cite{chicone1999evolution, colonius2012dynamics} for more on this subject.

\medskip

What is missing from the literature, however, is an extension of the Selgrade decomposition to the infinite-dimensional setting. The purpose of this paper is to address this gap, obtaining a Selgrade-type decomposition for linear semiflows of Banach space operators. The results in this paper are applicable to the derivative cocycles of a large class of dissipative parabolic equations.

In the finite-dimesional setting, the Selgrade decomposition sits between those of Sacker-Sell and Oseledets. To review, the Selgrade decomposition is the finest decomposition of the tangent bundle into continuous subbundles which are \emph{exponentially separated}: roughly, two subbundles are exponentially separated if over every point in the base space, the growth of vectors in one subbundle is exponentially larger than the growth of vectors in the other \cite{blumenthal2015characterization, bronsteinnonautonomous, colonius2012dynamics}. Equivalently, when viewed on projective space, exponentially separated subbundles correspond to attractor-repeller pairs, and so the Selgrade decomposition gives rise to the \emph{finest Morse decomposition} for the associated flow on the projective bundle; see \cite{colonius2012dynamics, selgrade1975isolated} for more details.

Exponential dichotomy is a strictly stronger condition than exponential separation, and so the Selgrade decomposition is a finer decomposition than that of Sacker-Sell (cf. \cite{sacker1978spectral, sacker1994dichotomies, selgrade1975isolated}). On the other side, the Selgrade decomposition can be thought of as a continuously-varying outer approximation to the Oseledets decomposition; this is especially useful due to the potential `irregularity' of the Oseledets decomposition \cite{colonius2012dynamics}.

\medskip

In a quite general infinite-dimensional setting, we are able to recover much of the finite-dimensional theory of Selgrade decompositions in this paper. Our results include (1) a characterization of exponentially separated subbundles as asymptotically compact attractor-repeller pairs for the semiflow on the projective bundle, and (2) an at-most countable decomposition into finite-dimensional exponentially separated subspaces. 

Everyone who builds an infinite-dimensional version of a finite dimensional theory is being punished twice: first, because proofs are very hard, and second, because, on the surface, the final product looks not much different from the original. This paper is not an exception. The usual difficulties that we must overcome are noncompactness of the infinite dimensional unit sphere, noninvertibility of injective linear maps, existence of subspaces with no direct complements, and presence of essential spectrum for infinite dimensional operators. 

In particular, our proof of (1) requires us to extend the theory of attractor-repeller pairs to the setting of semiflows on general metric spaces. Attractor theory in this setting is explored in \cite{hurley1995chain} (see also \cite{choi2002chain}). However, we are not aware of any previous detailed studies of repellers or attractor-repeller pairs for semiflows \emph{relative to the whole (non-locally compact) domain}. The closest approaches in the literature include studies of attractor-repeller pairs defined relative to compact invariant sets (see, e.g., \cite{rybakowski2012homotopy}); the literature on attractors for nonautonomous dynamical systems (see, e.g., \cite{carvalho2013attractors, kloeden2011nonautonomous} and the many references therein); and \cite{chen2011state}, where the authors define and briefly discuss a notion of repeller dual. These previous studies do not suffice for our purposes, and so in \S \ref{sec:generalAttractorTheory} we carefully develop a theory of repellers and attractor-repeller pairs for semiflows on general metric spaces when the attractor is asymptotically compact. 

We also rely on and further develop the techniques of \cite{blumenthal2015characterization} relating exponential splitting of cocycles and Gelfand $s$-numbers (a Banach space version of singular values; see \cite{pietsch1986eigenvalues} for a comprehensive review). This entails using some nontrivial facts regarding angles between infinite-dimensional subspaces used in \cite{Blumenthal2016} and $q$-dimensional volume growth used in \cite{Blumenthal20162377}.

\newpage
\subsection{Statement of results}

\subsubsection*{Assumptions} Let $B$ be a compact metric space with metric $d_B$. Let $\Bc$ be a real Banach space with norm $|\cdot|$; we write $\Vc = B \times \Bc$ for the trivial Banach bundle over $B$. At times, we will abuse notation somewhat and regard the fiber $\Vc_b = \{b\} \times \Bc$ over the point $b \in B$ as a vector space. We write $\pi_B : \Vc \to B$ for the projection onto $B$. We let $\phi : \R \times B \to B$ be a continuous flow on $B$. We write $\phi^t(\cdot) = \phi(t, \cdot)$ for the time-$t$ map of $\phi$.

%\begin{rmk}
%What do we lose if we consider only the trivial bundle $\Vc = B \times \Bc$, where $\Bc$ is a Banach space?
%\end{rmk}

%We let $(\psi_\a, U_\a)$ be a \emph{local trivialization} of $\Vc$; that is, $\{U_\a\}_{\a \in I}$ is a (finite) open cover of $\Vc$ of sets of the form $U_\a = \cup_{b \in V_\a} \Vc_b$, such that for each $\a \in I$, we have a homeomorphism $\psi_\a : U_\a \to V_\a \times \Bc_\a$, where $\Bc_\a$ is a Banach space with norm $|\cdot|_\a'$, and $\psi_\a$ sends fibers $\Vc_b$ to $\Bc_\a$ via a boundedly invertible linear transformation $\Psi_\a(b) : \Vc_b \to \Bc_\a$. We write $\psi_\a^2$ for the projection of $\psi_\a$ onto the second ($\Bc_\a$) coordinate.

\medskip

%\subsubsection*{Linear cocycles over $(B, \phi)$. } 
In all that follows, we assume that $\Phi : [0,\infty) \times \Vc \to \Vc$ is a semiflow on $\Vc$ of injective linear operators over $(B, \phi)$; that is, $\Phi$ is a semiflow on $\Vc$ for which
\begin{itemize}
\item[(H1)] $\pi_B \circ \Phi = \phi$; and
\item[(H2)] for any $(t, b) \in [0,\infty) \times B$, the map $v \mapsto \Phi(t, b, v)$ is a bounded, injective linear operator $\Vc_b \to \Vc_{\phi^t b}$.
\end{itemize}
For $t \geq 0, b \in B$, let us write $\Phi^t_b : \Vc_b \to \Vc_{\phi^t b}$ for the bounded, injective operator as in (b) above. We will assume that the assignment $(t, b) \mapsto \Phi^t_b$ satisfies the following continuity properties:
\begin{itemize}
\item[(H3)] For each fixed $t \geq 0$, the map $b \mapsto \Phi^t_b$ is continuous in the operator norm topology on $L(\Bc)$, the space of bounded linear operators on $\Bc$.
\item[(H4)] The mapping $(t, b) \mapsto \Phi^t_b$ is continuous in the strong operator topology on $L(\Bc)$.
\end{itemize}
{ As can be easily checked, property (H4) implies that $\Phi : [0,\infty) \times \Vc \to \Vc$ is a continuous mapping in the norm $d_\Vc$ on $\Vc$. Here $d_\Vc((b_1, v_1), (b_2, v_2)) := \max\{ d_B(b_1, b_2), |v_1 - v_2|\}.$}

\medskip

We write $\P \Vc$ for the projective bundle of $\Vc$, i.e., $\P \Vc = B \times \P \Bc$. Here, $\P \Bc$ is the projective space of $\Bc$, defined by $\P \Bc = (\Bc \setminus \{ 0 \}) / \sim$, where $v \sim w$ for $v, w \in \Bc \setminus \{ 0 \} $ iff $v = \lambda w$ for some $\lambda \in \R \setminus \{0 \} $. The metric $d_{\P \Vc}$ on $\P\Vc$ is now defined by 
\[
d_{\P \Vc}((b_1, v_1), (b_2, v_2)) = \max\{ d_B(b_1, b_2), d_\P (v_1, v_2) \} \, ,
\]
where $d_\P$ is the projective metric on $\P \Bc$ (defined in \eqref{eq:projectiveMetric}). The projectivized semiflow $\P \Phi : [0,\infty) \times \P \Vc \to \P \Vc$ is well-defined and continuous in the projective metric $d_{\P \Vc}$. { Note, however, that $\P \Phi$ need not be uniformly continuous in the $\P \Vc$ argument.}

\subsubsection*{Main results}

In the finite dimensional setting, it is well-known that attractor-repeller pairs for the projectivized flow { are in one-to-one correspondence with exponentially separated subbundles for the linear flow (see, e.g., Chapter 5 of \cite{colonius2012dynamics}). } 
Our first main result is an extension of this characterization to the infinite dimensional setting. Below the \emph{repeller dual} of an attractor $\Ac \subset \P \Vc$ for the 
projectivized semiflow $\P \Phi$ is denoted by $\Ac^*$; see \S \ref{subsec:repellerTheory} for a precise definition.

\begin{thmA}\label{thm:equiv}
Assume that $\Bc$ is a separable Banach space and that $B$ is chain transitive for the base flow $\phi$. Let $\Phi$ be a linear semiflow satisfying (H1) -- (H4) as above. Then, the following hold:
\begin{itemize}
	\item[(a)] Let $\Ac$ be an asymptotically compact attractor for $\P \Phi$, and write $\Ec := \P^{-1} \Ac$ and $\Fc := \P^{-1} \Ac^*$. Then, $\Ec, \Fc$ are continuous subbundles of $\Vc$ for which $\dim \Ec$ is finite and $\Vc = \Ec \oplus \Fc$. Moreover, this splitting is exponentially separated.
	\item[(b)] Let $\Vc = \Ec \oplus \Fc$ be a splitting into exponentially separated subbundles of $\Vc$ for which $\dim \Ec$ is finite and constant. Then $\P \Ec$ is an asymptotically compact attractor for $\P \Phi$ for which $(\P \Ec)^* = \P \Fc$.
\end{itemize}
\end{thmA}

{The definition of \emph{asymptotically compact attractor} is given precisely in Definition \ref{defn:asympCompact} (see also Definition \ref{defn:asymptoticallyCompact2}), although our usage here agrees with standard definitions in the literature (see \cite{carvalho2013attractors, hale2010asymptotic, sell2013dynamics}).} Exponential separation is defined in \S \ref{subsec:exponentialSplittings}. The proof of Theorem \ref{thm:equiv} is an adaptation to the infinite-dimensional setting of the finite-dimensional version presented in \cite{salamon1988flows} and \cite{colonius2012dynamics}.

We note that it is entirely possible for a compact attractor of $\P \Phi$ to fail to be asymptotically compact, as the following example shows.
\begin{ex}\label{ex:one}
We construct a bounded linear operator on $\ell^2(\N)$ as follows. Denote by $\{e_n\}_{n = 1}^\infty$ the standard basis for $\ell^2(\N)$. For each $t \geq 0$ we now define the bounded linear operator $T^t : \ell^2(\N) \to \ell^2(\N)$ by $T^t e_1 = e_1$ and $T^t e_n = (\frac{n - 1}{n})^t e_n$ for $n > 1$. 
%\[
%T(a_1, a_2, \cdots) = (a_1, \frac12 a_2, \frac23 a_3, \cdots, \frac{n-1}{n} a_n, \cdots) \, .
%\]
Note that although $\Ac = \{\P e_1\}$ is an attractor for $\P T : \P \ell^2(\N) \to \P \ell^2(\N)$, the subspace $\operatorname{Span}\{e_1\}$ is not exponentially separated from its orthogonal complement.
\end{ex}

 {We note, however, that in the above example the operator $T$ is not compact. Indeed, were $T$ any injective, compact linear operator and $\Ac$ a compact attractor for $\P T$, then it is a simple exercise to show that any compact attractor $\Ac$ would be automatically asymptotically compact. In Example \ref{ex:one}, any compact attractor for $\P T$ is a finite sum of generalized eigenspaces. The authors are not aware of an answer to the following question: \emph{If $\Phi$ is a linear semiflow of injective compact linear operators as in (H1) -- (H4), then is it possible for a compact attractor of $\P \Phi$ to fail to be asymptotically compact as in Definition \ref{defn:asymptoticallyCompact2}?}}

\medskip

Our second main result is a generalization of the classical Selgrade decomposition for linear flows on a finite dimensional vector bundle: \emph{a 
(finite) finest Morse decomposition (equivalently, a finest attractor sequence) of the projectivized flow exists} \cite{selgrade1975isolated}. Here, we will obtain a (at-most countable) finest attractor sequence comprised of \emph{asymptotically compact attractors}. 

%Another aspect of the Selgrade decomposition regards the structure of the chain recurrent set for the projectivized flow; equivalently, Selgrade's Theorem says that a finest Morse 
%decomposition of the projectivized flow exists and is finite.

%In the infinite dimensional setting it is no longer reasonable to expect that a \emph{finite} finest Morse decomposition should exist. Instead, as we show, under the conditions of our previous result we obtain a finest Morse decomposition which is at most countable.

\begin{thmA}\label{thm:finestDecomp}
Assume that $\Bc$ is a separable Banach space, and that $B$ is chain transitive for the base flow $\phi$. Let $\Phi$ be a linear semiflow as in (H1) -- (H4) above.

Then, there is an at-most countable sequence $\{\Ac_i\}_{i = 0}^N, N \in \Z_{\geq 0} \cup \{\infty\}$, of subsets of $\P \Vc$, with $\Ac_0 = \emptyset$ and $\Ac_i \subset \Ac_{i +1}$ for all $0 \leq i < N$, with the following properties:
\begin{itemize}
\item[(a)] For any $1 \leq i < N + 1$\footnote{What is meant by this shorthand is that if $N = \infty$, then $i \in \N$, and if $N < \infty$, then $1 \leq i \leq N$}, we have that $\Ac_i$ is an asymptotically compact attractor for $\P \Phi$.
\item[(b)] The sequence $\{\Ac_n\}$ is the finest such collection in the following sense: if $\Ac$ is any nonempty asymptotically compact attractor for $\P \Phi$, then $\Ac = \Ac_i$ for some $1 \leq i < N + 1$.
\end{itemize}
\end{thmA}
The proof of Theorem \ref{thm:finestDecomp} uses characterization of asymptotically compact attractors for $\P \Phi$ in Theorem \ref{thm:equiv} in addition to the characterization of exponential separation given in \cite{blumenthal2015characterization}, which we recall in Theorem \ref{thm:characterizeExpSep}, and a certain induction-type result {(Proposition \ref{prop:moveDownSubspace})} for exponentially separated subbundles which may be of independent interest.

With $\{\Ac_i\}, N$ as in Theorem \ref{thm:finestDecomp}, write $\Vc_i^+ = \P^{-1} \Ac_i = \Vc_1 \oplus \cdots \oplus \Vc_i$ and $\Vc_i^- = \P^{-1} \Ac_i^*$ for each $1 \leq i < N + 1$ so that $\Vc = \Vc_i^+ \oplus \Vc_i^-$ is an exponentially separated splitting of $\Vc$. We also write $\Vc_i = \Vc_i^+ \cap \Vc_{i - 1}^-$ and $\Mc_i =  \P \Vc_i = \Ac_i \cap \Ac_{i - 1}^*$; by Theorem \ref{thm:finestDecomp}, each $\Vc_i$ is a finite dimensional, equivariant, continuous\footnote{See Lemmas \ref{lem:compactSubbundle} and \ref{lem:fcContinuous}. We note however that continuity can be deduced directly from exponential separation, as carried out in, e.g., \cite{blumenthal2015characterization}} subbundle of $\Vc$. 

\begin{defn}
We call the subbundles $\{\Vc_i\}_{i = 1}^N$ the {\bf discrete Selgrade decomposition}\footnote{We use the terminology `discrete' to evoke an analogy with the discrete spectrum of a closed linear operator.} of $\Phi$. 
\end{defn}

We note that in Theorem \ref{thm:finestDecomp} it is possible for $\P \Phi$ to admit no asymptotically compact attractors. This stands in contrast to the finite dimensional case, where the Selgrade decomposition $\{\Vc_i\}$ may be trivial in the sense that $\Vc_1 = \Vc$, hence $\P \Vc$ is chain transitive under $\P \Phi$ (c.f. Corollary \ref{cor:chainTrans} below).

\begin{rmk}\label{rmk:subbundleFormulation}
It is possible to formulate the preceding results for more general bundles $\Vc$ than the trivial bundle. For simplicity, however, we do not pursue these extensions here, except to note that everything we do holds with virtually no changes when $\Vc$ is replaced with a continuously-varying finite-codimensional subbundle $\hat \Vc$ of the trivial bundle $\Vc = B \times \Bc$. That is, each fiber $\hat \Vc_b$ over $b \in B$ is a closed, finite-codimensional subspace of $\Bc$, and $b \mapsto \hat \Vc_b$ varies continuously in the Hausdorff distance (see \S \ref{subsec:grassmanian} for definitions).
\end{rmk}

The following corollaries describe additional properties of the discrete Selgrade decomposition $\{\Vc_i\}$.
\begin{cor}\label{cor:chainTrans}
Assume the setting of Theorem \ref{thm:finestDecomp}. For each $1 \leq i < N + 1$, the set $\Mc_i = \P \Vc_i$ is a chain transitive set for the projectivized flow $\P \Phi$.
\end{cor}
Corollary \ref{cor:chainTrans} follows from Theorem \ref{thm:finestDecomp} and the classical Conley theory applied to the linear flow $\P \Phi|_{\P \Vc_i^+}$ for $1 \leq i < N + 1$. { This falls entirely under the purview of the finite-dimensional theory, an so details are left to the reader (see, e.g., \cite{colonius2012dynamics}).}

Note, however that we do not make any claim on the structure of chain recurrent points in $\P \Vc_i^-$. Indeed, the components $\{\Mc_i\}$ of the discrete
Selgrade decomposition need not contain all chain recurrent points for $\P \Phi$.

\begin{ex}\label{ex:two}
In the notation of Example \ref{ex:one}, for each $t \geq 0$ define the bounded linear operator $S^t : \ell^2(\N) \to \ell^2(\N)$ defined by setting $S^t e_1 = e_1$ and $S^t e_n = (\frac{n}{n-1})^t e_n$ for $n > 1$.
%\[
%T(a_1, a_2, \cdots) = (a_1, 2 a_2, \frac32 a_3, \frac43 a_4, \cdots, \frac{n}{n-1} a_n, \cdots) \, .
%\]
In the notation of Theorem \ref{thm:finestDecomp}, we have $N = \infty$, and $\Mc_i = \{\P e_{i + 1}\}$ for each $1 \leq i < \infty$. On the other hand, 
$e_1$ is an eigenvector with eigenvalue $1$, hence $\P e_1$ is a chain recurrent point for $\P S^t$ as  not contained in $\cup_{i = 1}^\infty \Mc_i$.
\end{ex}

Note that the operators $S^t$ in Example \ref{ex:two} are noncompact (indeed, the eigenvalue $1$ sits on the boundary of the essential spectrum). { Below we give an example of a family $\{ T_b \}_{b \in B}$ of injective, compact linear maps for which many chain recurrent points exist while admitting no forward invariant finite-dimensional subbundle.

\begin{ex}
Let $B = [1/2, 1]$ and let $\phi : B \to B$ be the identity map. For $b \in B$, define the operator $T_b$ on $\ell^2(\N)$ by $T_b e_n = n^{-1} \big( b e_n + (1 - b) e_{n + 1} \big) $ for all $n \geq 1$. As one can check, $b \mapsto T_b$ is continuous in the operator norm on $\ell^2(\N)$ and $T_b$ is injective for any $b \in B$. Moreover, points of the form $(1, \P e_n) \in \P \Vc$ are recurrent for the linear flow for any $n$ (indeed, each is a fixed point), yet $T_b$ admits no forward invariant finite-dimensional subspace for any $b \in [1/2, 1)$.
\end{ex}

On the other hand, $\{ T_b\}_{b \in B}$ cannot be realized as the time-one map of a semiflow. The authors are not aware of an answer to the following question: \emph{Is it possible to construct a linear semiflow of compact operators satisfying (H1) -- (H4) for which $N = 0$ in Theorem \ref{thm:finestDecomp} while admitting chain recurrent points?}
}

%Let $B = \prod_{n = 1}^\infty [1/(n + 1), 1]$ and let $\phi : B \to B$ be the identity map. Let $\Bc = \ell^2(\N)$ and for $b = (b_1, b_2, \cdots) \in B$, define $T_b e_i = i^{-2} \big( b_i e_i + (1 - b_{i}) e_{i + 1} \big) $ for all $i \geq 1$. Then, $b \mapsto T_b$ is continuous in the operator norm on $\ell^2(\N)$ and $T_b$ is injective for any $b \in B$. Moreover, $((1,0,\cdots), \P e_1) \in \P \Vc$ is recurrent for the linear flow (indeed, a fixed point) and yet $T$ admits no forward invariant finite-dimensional subbundle.

%\begin{conj}
%Assume the setting of Theorem \ref{thm:finestDecomp}, and additionally, that $\Phi^t_b$ is a compact linear operator for each $t \geq 0, b \in B$. Then, in the notation of Theorem \ref{thm:finestDecomp}, we have that $\cup_{i = 1}^N \Mc_i$ is the maximal chain recurrent set for $\P \Phi$.
%\end{conj}
%{\color{blue} It is an open question whether $\cup_{i = 1}^\infty \Mc_i$ is the chain recurrent set when $\Phi$ is a linear semiflow comprised of compact operators.}

Our second corollary pertains to the {\bf discrete Morse spectrum} associated with the discrete Selgrade decomposition given earlier. Given a (compact) chain transitive component $\Mc$ for $\P \Phi$, we define the Morse spectrum $\Sigma_{\text{Mo}}(\Phi; \Mc)$ by
\begin{align*}
\Sigma_{\text{Mo}}(\Phi; \Mc) = \{ \lambda \in \R : & \,  \text{there are } \e^k \to 0, T^k \to \infty \text{ and } (\e^k, T^k) \text{ chains } \zeta^k \text{ in }\P \Vc \\
&  \text{ such that } \lambda(\zeta^k) \to \lambda \text{ as } k \to \infty\} \, .
\end{align*}
Chains are defined in \S \ref{subsec:chains}. Here, when $\zeta = \{(b_i, \P v_i)\}_{i = 0}^{n-1} \subset \P \Vc, \{T_i\}_{i = 0}^{n-1}$ is a chain for $\P \Phi$, we have written
\[
\lambda(\zeta) = \bigg( \sum_{i = 0}^{n-1} T_i \bigg)^{-1} \bigg( \sum_{i = 0}^{n-1} \log |\Phi^{T_i}_{b_i} v_i| \bigg)  \, ,
\]
where for each $i$, $v_i \in \Vc_{b_i}$ is a unit vector representative of $\P v_i$.

\begin{defn}
The discrete Morse spectrum $\Sigma_{{\rm Mo}}^{{\rm dis}}(\Phi)$ for $\Phi$ is defined by
\[
\Sigma_{{\rm Mo}}^{{\rm dis}}(\Phi) = \bigcup_{i = 1}^N \Sigma_{{\rm Mo}} (\Phi; \Mc_i) \, .
\]
\end{defn}

We now state the following description of the discrete Morse spectrum $\Sigma_{{\rm Mo}}^{{\rm dis}}(\Phi)$. Below, the Lyapunov exponent $\lambda(b, v)$ of a point $(b, v ) \in \Vc, v \neq 0$, is defined by $\lambda(b, v) := \limsup_{t \to \infty} t^{-1} \log |\Phi^t_b v|$. 

\begin{cor}\label{cor:morseSpec}
For each $1 \leq i < N + 1$, the Morse spectrum of $\Phi|_{\Vc_i}$ is a compact interval of the form $\Sigma_{{\rm Mo}}(\Phi; \Mc_i) = [\kappa^*(\Mc_i), \kappa(\Mc_i)]$, where $\kappa^*(\Mc_i), \kappa(\Mc_i)$ are attained Lyapunov exponents of $\Phi$, and for $1 \leq i < N$, we have
\[
\kappa^*(\Mc_i) < \kappa^*(\Mc_{i + 1}) \quad \text{ and } \quad \kappa(\Mc_i) < \kappa(\Mc_{i + 1}) \, .
\]
Moreover, the Lyapunov spectrum $\Sigma_{{\rm Lyap}}(\Phi; \Mc_i) = \{ \lambda(b, v) : (b, v) \in \Vc_i, v \neq 0\}$ is contained in $\Sigma_{{\rm Mo}}(\Phi; \Mc_i)$.
\end{cor}
\noindent Corollary \ref{cor:morseSpec} follows from the finite dimensional analogue applied to $\Phi|_{\Vc_i^+}$. {Again this fits in the framework of the finite-dimensional theory (see, e.g., \cite{colonius2012dynamics}), and so details are left to the reader.}

\medskip

So far we have not discussed the Morse spectrum associated to the `essential' Selgrade subbundle $\Vc^- := \cap_{i = 1}^N \Vc_i^-$. This subbundle can easily fail to be chain transitive, and so it is possible that $\Sigma_{{\rm Mo}}(\Phi ; \Vc^-)$ need not be a connected interval. Moreover, it is possible for $\Sigma_{{\rm Mo}}(\Phi ; \Vc^-)$ 
to overlap with $\Sigma^{{\rm dis}}_{{\rm Mo}}(\Phi)$, as the following example illustrates. 
%As before, Corollary \ref{cor:morseSpec} is unable to capture the entire Lyapunov spectrum of $\Phi$ inside the discrete Morse spectrum $\Sigma_{{\rm Mo}}^{{\rm exp}}(\Phi)$.

\begin{ex}\label{ex:three}
In the notation of Example \ref{ex:two}, let $B = [1,2]$ equipped with the identity map and consider the linear semiflow $S^t_b := b^t S^t$ over $B$. Then in the notation of Theorem \ref{thm:finestDecomp} we have $N = \infty$, $\Mc_i = \{\P e_{i + 1}\}$ and $\Vc^- = \operatorname{Span} \{e_1\}$. Moreover $\Sigma_{{\rm Mo}}^{{\rm dis}}(\Phi) = (1,4]$, while $\Sigma_{{\rm Mo}}(\Phi; \Vc^-) = [1,2]$.
\end{ex}

\subsubsection*{Plan for the paper}

The plan for the paper is as follows. In \S \ref{sec:generalAttractorTheory} we recall elements of the theory of asymptotically compact attractors for semiflows on a general metric space. 
Much of this is review, although the material in \S \ref{subsec:repellerTheory} on repeller duals does not, to the knowledge of the authors, appear elsewhere in the literature. In \S \ref{sec:banachPrelim} we recall necessary preliminaries from Banach space geometry.

The bulk of the original work in this paper is devoted to proving the `(a)$ \Rightarrow$ (b)' implication in Theorem \ref{thm:equiv}. This is proved as Proposition \ref{prop:compSubspace} in \S \ref{sec:hard}. We complete the proofs of Theorems \ref{thm:equiv} and \ref{thm:finestDecomp} in \S \ref{sec:cleanUp}.

\section{Attractors for semiflows on general metric spaces}\label{sec:generalAttractorTheory}

\subsubsection*{Setting for \S \ref{sec:generalAttractorTheory}.} For the purposes of this section, we let $X$ be a complete metric space with metric $d_X$. Throughout, for $x \in X$ and $r > 0$ we write $B_r(x) = \{y \in X : d_X(x, y) < r\}$ for the open ball of radius $r$ centered at $x$. For $r > 0$ and a set $Z \subset X$, we write $B_r(Z) = \{y \in X : d_X(y, Z) < r\}$, where here $d_X(y, Z) = \inf\{ d_x(y, z) : z \in Z\}$ denotes the minimal distance from $y$ to $Z$.

Through this section we study an \emph{injective, continuous semiflow} $\psi$ on $X$; that is, $\psi : [0,\infty) \times X  \to X$ is a continuous map for which (i) $\psi(0, x) = x$, (ii) $\psi(t, \psi(s, x)) = \psi(t + s, x)$ for all $s, t \geq 0$, $x \in X$, and (iii) for all $t \geq 0$, the map $x \mapsto \psi(t,x)$ is injective. We emphasize that the time-$t$ maps $\psi^t(\cdot) := \psi(t, \cdot)$ are not assumed to be  invertible on all of $X$.

\subsection{Preattractors and attractors}\label{subsec:preattractors}

Much of the material in \S \ref{subsec:preattractors} is standard (see \cite{hurley1995chain}). However, due to its importance in the formulation of the results of this paper and the arguments to come, we review it in detail.

Following Hurley \cite{hurley1995chain}, we make the following definition. 

\begin{defn}
A nonempty open set $U \subset X$ is called a \emph{preattractor} if for some $T > 0$, we have that 
\[
\overline{\psi([T, \infty) \times U)} \subset U \, .\]
 We associate to preattractors $U$ a corresponding \emph{attractor} $A$, defined by
\[
A = \bigcap_{t \geq 0} \overline{\psi([t, \infty) \times U)} \, .
\]
We refer to the pair $(A, U)$ as an \emph{attractor pair}. Note that
\[
A =  \omega(U) := \{ \text{limits of the form } \lim_n \psi^{t_n} x_n \, , \text{ where } t_n \to \infty \text{ and } \{x_n\} \subset U \} \, .
\]
\end{defn}

We note that this differs from the classical definition of `attractor'; when, however, $X$ is a compact metric space, this definition coincides with the usual one. The definition of preattractor was introduced in \cite{hurley1995chain} (see also \cite{choi2002chain}), where it was used to characterize chain recurrence for flows on noncompact spaces.

\begin{lem}\label{lem:attractor1}
Let $(A, U)$ be an attractor pair, and let $x \in X$ be such that $\omega(x) \cap A \neq \emptyset$. Then, $\omega(x) \subset A$.
\end{lem}
\begin{proof}
Let $x^* \in \omega(x) \cap A$, and let $t_n \to \infty$ be a sequence for which $\psi^{t_n} x \to x^*$. Then, for $n$ sufficiently large we have $\psi^{t_n} x \in U$, hence $\omega(x) \subset \omega(U) = A$.
\end{proof}

\begin{defn}\label{defn:asympCompact}
An attractor pair $(A, U)$ is \emph{asymptotically compact} if the following holds: for any sequence of reals $t_n \to \infty$ and any sequence $\{x_n\} \subset U$, we have that $\{\psi^{t_n} x_n\}$ possesses some convergent subsequence.
\end{defn}

Note that an attractor may be empty, whereas an asymptotically compact attractor is always nonempty (Lemma \ref{lem:asympCompact} below). Even when $A$ is nonempty and compact, a preattractor $U$ for $A$ may contain points which have empty $\omega$-limit sets, running contrary to the typical intuition that attractors genuinely `attract' an open neighborhood of initial conditions. Asymptotic compactness precludes this possibility.

\medskip

{The concept of asymptotic compactness is prominent in the study of infinite-dimensional dissipative dynamical systems \cite{hale2010asymptotic}, where it is often used to check for the existence of a maximal global attractor (see, e.g., \cite{carvalho2013attractors, sell2013dynamics}). However, the standard definition usually refers to a property of the semiflow itself: the semiflow $\{ \phi^t\}$ is called asymptotically compact if Definition \ref{defn:asympCompact} holds with $U = X$. As the following example shows, asymptotic compactness in this sense need not hold for a projectivized linear semiflow, even when the linear semiflow consists of compact operators.

%The following example illustrates the importance of considering attractor pairs:

\begin{ex}\label{ex:badAsympCompact}
Consider the semiflow $\psi^t = \P T^t$ on $\P \ell^2(\N)$, where for $t \geq 0$, $T^t : \ell^2(\N) \to \ell^2(\N)$ is the compact linear operator defined by
\[
%T(a_1, a_2, \cdots) = ( a_1, 0, a_2,  a_3, \cdots, a_n, \cdots) \, .
T^t e_n = n^{-t} e_n \, .
\]
Observe that $\P \ell^2(\N)$ is (trivially) a preattractor with corresponding attractor $A_\infty = \{ \P e_n \}_{n \geq 1}$. The attractor pair $(A_\infty, \P \ell^2(\N))$ is \emph{not} asymptotically compact, since $\{\P T^n (\P e_n)\}_n = \{ \P e_n\}_n$ has no convergent subsequence in $\ell^2(\N)$. 

Let $A_1 = \{ \P e_1\}  \subset \P \ell^2(\N)$, and let $U_1$ be a small open neighborhood of $A_1$. Then, as one can check, $(A_1, U_1)$ is an asymptotically compact attractor pair. 
\end{ex}
}
%\begin{rmk}
%It is possible for two preattractors $U, U'$ to correspond to the same attractor $A$ and yet have one pair $(A, U)$ be asymptotically compact and the other pair $(A, U')$ not. 
%
% is for this reason that we insist on referring to attractor pairs $(A, U)$ when dealing with general attractors.
%Thus asymptotic compactness is a property of attractor pairs, not attractors alone. { amend this!}
%\end{rmk}

Asymptotically compact attractor pairs have many qualities similar to their counterparts in the locally compact setting.
\begin{lem}\label{lem:asympCompact}
Let $(A, U)$ be an asymptotically compact attractor pair. Then,
\begin{itemize}
	\item[(a)] we have that $A$ is nonempty and compact;
	\item[(b)] for any $x \in U$, $\omega(x)$ is nonempty and $\omega(x) \subset A$; and
	\item[(c)] for any $t \geq 0$ we have $\psi^t(A) = A$.
\end{itemize}
%As a consequence, $\psi^t|_A$ acts as a homeomorphism on $A$ for any $t \geq 0$.
\end{lem}
\begin{proof}
Item (b) is immediate. For (c) we use the fact that $A = \omega(U)$: for any $t > 0$ we will show that $\psi^t \omega(U) = \omega(U)$. The inclusion ``$\subset$'' is easiest: if $\{x_n\} \in U, t_n \to \infty$ are sequences for which $\psi^{t_n}(x_n) \to x$ for some $x \in \omega(U)$, then $\psi^t x = \lim_{n \to \infty} \psi^{t + t_n}x_n$ by continuity, hence $\psi^t x \in \omega(U)$.

For the other direction, fix $x \in \omega(U)$ and let $\{x_n\} \subset U, t_n \to \infty$ be such that $x = \lim_{n \to \infty} \psi^{t_n}x_n$. For $n$ sufficiently large we have $t_n \geq t$, and so $\psi^{t_n - t}(x_n)$ is defined for $n$ sufficiently large. Applying asymptotic compactness, let $x^* \in \omega(U)$ be a subsequential limit point of $\{\psi^{t_n - t} x_n\}$. Again by continuity, we have that $\psi^t x^* = \lim_n \psi^{t_n} x_n = x$, hence $x \in \psi^t \omega(U)$.

To show (a), recall that $A$ is nonempty by the definition of asymptotic compactness. It remains to prove that $A$ is sequentially compact. To see this, let $\{x_n\} \subset A$ be any infinite sequence, and let $t_n \to \infty$ be arbitrary. Writing $ x_n' = \psi^{- t_n} x_n$, using (c) to do so, it follows that $\{\psi^{t_n} x_n'\} = \{x_n\}$ possesses a subsequential limit in $\omega(U) = A$.
\end{proof}

We now prove the following useful characterization of asymptotic compactness for semiflows on metric spaces.

\begin{lem}\label{lem:asympCompact2}
Let $(A, U)$ be an attractor pair. Then, the following are equivalent.
\begin{itemize}
\item[(a)] $(A, U)$ is asymptotically compact.
\item[(b)] $A$ is nonempty, compact, and for any $\e > 0$ there exists $T > 0$ such that $\overline{\psi([T, \infty) \times U )} \subset B_\e(A)$.
\end{itemize}
\end{lem}
\begin{proof}
\noindent {\bf (a) $\Rightarrow$ (b). } That $A$ is nonempty and compact was established in Lemma \ref{lem:asympCompact}. Assume now the following contradiction hypothesis: there exists some $\e > 0$ such that for any $T > 0$, we have that $\overline{\psi([T, \infty) \times U )} \setminus B_\e(A) \neq \emptyset$.

For each $T \in \N$, fix sequences $x_n^{(T)} \subset U$ and $t_n^{(T)} \geq T$ converging to an element of $\overline{\psi([T, \infty) \times U)} \setminus B_\e(A)$. Then, there is an $N^{(T)} \in \N$ sufficiently large so that $\psi^{t_n^{(T)}} x_n^{(T)} \notin B_{\e / 2}(A)$ for all $n \geq N^{(T)}$. Define now the diagonal subsequences $\tilde t_L = t_{N^{(L)}}, \tilde x_L = x_{N^{(L)}}^{(L)}$ and note that $\tilde t_L \to \infty$ as $L \to \infty$. On the other hand, the limit points of $\{\psi^{\tilde t_L} \tilde x_L\}$ (of which there is at least one, by asymptotic compactness) are at distance $\geq \e/2$ from $A$, which contradicts the definition $A = \omega(U)$. Thus (b) holds.

\noindent{\bf (b) $\Rightarrow$ (a). } Let $\{x_n\} \subset U$ and $t_n \to \infty$. Using (b), it follows that for any $\e > 0$ there exists $N = N(\e)$ such that for any $n \geq N$, we have that $d_X(\psi^{t_n} x_n , A) < \e$ for all $n \geq N$. Define the subsequence $n_i = N(1/i)$, and for each $i$ let $\hat x_i \in A$ be such that $d_X(\psi^{t_{n_i}} x_{n_i}, \hat x_i) < 2 / i$. Then, by compactness the sequence $\{\hat x_i\}_i$ has a convergent subsequence, which by construction is a cluster point of $\{\psi^{t_{n_i}} x_{n_i}\}$, hence of $\{\psi^{t_n} x_n\}$. This completes the proof.
\end{proof}

{ Lemma \ref{lem:asympCompact2} has the following consequence: given an asymptotically compact attractor pair $(A, U)$, there exists $\e > 0$ sufficiently small so that $B_\e(A)$ is a preattractor for $A$ for which $(A, B_\e(A))$ is asymptotically compact (cf. Example \ref{ex:badAsympCompact}). Thus we obtain the following `intrinsic' formulation of the asymptotic compactness property.
%It suffices to choose $\e > 0$ small enough so that $B_\e(A) \subset U$-- such an $\e$ always exists by the compactness of $A$. 

\begin{defn}\label{defn:asymptoticallyCompact2}
A compact, forward invariant subset $A \subset X$ is an asymptotically compact attractor if for some (hence all sufficiently small) $\e > 0$ we have that $(A, B_\e(A))$ is an asymptotically compact attractor pair.
\end{defn}
}

\subsection{Repellers and attractor-repeller duals}\label{subsec:repellerTheory}

%From here on, we assume the following:
%\begin{align}\label{eq:semiflowInj}
%\text{For any $t \geq 0$, the map $\psi^t : X \to X$ is injective.}
%\end{align}
Here we discuss repellers and repeller-duals in our noncompact, noninvertible setting. To the best of the authors' knowledge, the material in \S \ref{subsec:repellerTheory} 
does not appear elsewhere in the literature. For the closest alternative approach, we refer to the book of Rybakowski \cite{rybakowski2012homotopy}, where the attract-repeller theory is recovered for semiflows restricted to compact invariant sets. In comparison, we present here an attractor-repeller theory that does not restrict to compact invariant sets, and instead is carried out on the entire space $X$. 

\begin{defn}\label{defn:prerepeller}
A \emph{prerepeller} is a nonempty open set $V \subset X$ with the property that for some $T > 0$, we have that $\overline{\cup_{t \geq T} (\psi^t)^{-1}(V)} \subset V$. The repeller $R$ associated to a prerepeller $V$ is defined to be
\[
R = \bigcap_{t \geq 0} \overline{\bigcup_{s \geq t} (\psi^s)^{-1}(V)} \, .
\]
Above, $(\psi^s)^{-1}(V)$ refers to the preimage of $V$. We call $(R, V)$ a \emph{repeller pair}. Note that $R$ may be empty.
\end{defn}

{ We give an alternative \emph{limit set} characterization of $R$ as follows. Let us abuse notation and write $\psi^{-t} x \in X$ for the preimage $(\psi^{t})^{-1} \{ x \}$; by injectivity, $\psi^{-t} x \in X$ is defined when it exists. Then,
\[
R = \omega^*(V) := \bigg\{ \begin{array}{c} \text{limits of the form } \lim_{n \to \infty} \psi^{-t_n} x_n \, , \text{ where } \{ x_n \} \subset V, \\ \, t_n \to \infty \, , \text{ and } \psi^{-t_n} x_n \text{ exists for each } n \end{array} \bigg\} \, .
\]
}
\begin{lem}\label{lem:repellerInvariant}
Let $(R, V)$ be a repeller pair. Then, $R$ is a closed, possibly empty, set for which $\psi^t R \subset R$ for all $t \geq 0$.
\end{lem}
\begin{proof}
We compute
\begin{align*}
\psi^t(R) & = \psi^t \bigg(\bigcap_{T \geq 0} \overline{\bigcup_{s \geq T} (\psi^s)^{-1}(V)}  \bigg) = \psi^t \bigg(\bigcap_{T \geq t} \overline{\bigcup_{s \geq T} (\psi^s)^{-1}(V)}  \bigg)  \\
& \subset \bigcap_{T \geq t} \psi^t \bigg( \overline{\bigcup_{s \geq T} (\psi^s)^{-1}(V)} \bigg)  
 \subset \bigcap_{T \geq t} \overline{\bigcup_{s \geq T} (\psi^{s - t})^{-1} (V) } = R \, , 
\end{align*}
having used the continuity of the time-$t$ map $\psi^t$ to deduce that $\psi^t(\overline Y) \subset \overline{\psi^t(Y)}$ for any subset $Y \subset X$.
\end{proof}
\noindent Note that the inclusion in Lemma \ref{lem:repellerInvariant} may be strict (contrast with Lemma \ref{lem:asympCompact}).

\medskip

{ We now turn our attention to the duality between attractors and repellers in our setting, assuming asymptotic compactness of the attractor.
}{ \begin{defn}
Let $C \subset X$. We define the \emph{dual} $C^*$ of $C$ to be
\[
C^* = \{ x \in X : \omega(x) \cap C = \emptyset \} \, .
\]
\end{defn}
}
\begin{lem}\label{lem:dual}
Let $(A, U)$ be an asymptotically compact attractor pair. Then, $A^*$ is the repeller corresponding to the prerepeller { $V$ defined by
\begin{align}\label{eq:definePrerepellerV}
V := X \setminus \overline{\psi([T, \infty) \times U)} \, ,
\end{align}
where $T \geq 0$ is as in the definition of preattractor for $U$ (i.e., $\overline{\psi([T, \infty) \times U)} \subset U$).} In particular, $A^*$ is closed. Moreover we have $A \cap A^* = \emptyset$.
\end{lem}
\noindent In light of Lemma \ref{lem:dual}, we are justified in referring to $A^*$ as the {\bf repeller dual} of $A$.
\begin{proof}
Note that $V$ is open and $X = U \cup V$. We claim that $V$ is a prerepeller with repeller $A^*$.

We first show that $V$ is a prerepeller; it suffices to show that $\overline{\cup_{t \geq T} (\psi^t)^{-1}(V)} \subset V$, where $T$ is as above. For this, let $\{v_n\} \subset \cup_{t \geq T} (\psi^t)^{-1}(V)$ be a sequence converging to a point $v \in \overline{\cup_{t \geq T} (\psi^t)^{-1} (V)}$; for each $n$ let $t_n \geq T$ be such that $\psi^{t_n} v_n \in V$. 

If $v \notin V$, then $v \in U$, and so $v_n \in U$ for $n$ sufficiently large. But then $\psi^{t_n} v_n \in \psi^{t_n} U \subset \overline{\psi([T, \infty) \times U)}$, contradicting the assumption that $\psi^{t_n} v_n \in V$ for all $n$. Thus all such limit points $v$ belong to $V$, and we conclude that $V$ is a prerepeller.

Let $R$ be the repeller corresponding to $V$; we now show that $R = A^*$. To show $R \subset A^*$, let $v \in R$ and assume for the sake of contradiction that $\omega(v) \cap A \neq \emptyset$. Then there is a sequence of times $t_n \to \infty$ for which $\psi^{t_n} v \to x^*$ for some $x^* \in A$. In particular, $\psi^{t_n} v \in U$ for $n$ sufficiently large, and so $\psi^{t_n} v \in \overline{\psi([T, \infty) \times U)}$ on taking $n$ sufficiently larger. We conclude that $\psi^{t_n} v \notin V$ for such $n$. On the other hand, by Lemma \ref{lem:repellerInvariant}, $\psi^t v \in R \subset V$ for all $t$, and so we have a contradiction. Thus $R \subset A^*$.

To show $A^* \subset R$, let $v \in A^*$. Observe, then, that $\psi^t v \notin U$ for any $t \geq 0$ by asymptotic compactness-- otherwise, $\omega(v) \cap A$ would be nonempty by asymptotic compactness. It follows that $\psi^t v \in V$ for all $t \geq 0$, i.e., $v \in (\psi^t)^{-1}(V)$ for all $t \geq 0$. Thus $v \in R$ by construction; this completes the proof of $A^* = R$.

The fact that $A \cap A^* = \emptyset$ follows from the fact that $A^* \subset V$ as above and that $A \subset \overline{\psi([T, \infty) \times U)} = X \setminus V$.
\end{proof}
%\begin{rmk}
%The inclusion $A^* \subset R$ is where asymptotic compactness is really used.
%\end{rmk}

\subsubsection*{Properties of the repeller dual}

Although $A^*$ may be empty, the exterior of a neighborhood of $A$ always `attracts' trajectories in backwards time in the sense of preimages.
\begin{lem}\label{lem:repellersRepel}
Let $(A, U)$ be an asymptotically compact attractor pair, and let $\e > 0$ be sufficiently small. Define $V_\e = \{x \in X : d(x, A) > \e\}$. Then, $V_\e$ is a prerepeller, and $(A^*, V_\e)$ is a repeller pair.%, i.e., there exists $T > 0$ such that
%\[
%\overline{\bigcup_{t \geq T} (\psi^t)^{-1} V_\e } \subset V_\e \, ; 
%\]

\end{lem}
\begin{proof}
Fix $\e > 0$ sufficiently small so that $B_{2 \e} (A) \subset U$. To show that $V_\e$ is a prerepeller, assume not for the sake of contradiction: that is, for any $T > 0$, $\overline{\cup_{t \geq T} (\psi^t)^{-1} V_\e } \setminus V_\e \neq \emptyset$. It follows that{  $B_{2 \e}(A) \cap \cup_{t \geq T} (\psi^t)^{-1} V_\e \neq \emptyset$} for all $T$, and so there exists a sequence $t_n \to \infty$ and points $x_n \in B_{2 \e}(A) \subset U$ for which $\psi^{t_n} x_n \notin B_\e(A)$ for all $n$. This contradicts the asymptotic compactness of $(A, U)$.

{ We now check that the repeller 
\[
R_\e = \bigcap_{t \geq 0} \overline{\bigcup_{s \geq t} (\psi^s)^{-1} V_\e }
\]
does, indeed, coincide with $A^*$. To check $R_\e \subset A^*$,  assume that there exists an element $x \in R_\e \setminus A^*$. Let $x = \lim_n \psi^{-t_n} x_n$, where $\{ x_n \} \subset V_\e$ and $t_n \to \infty$. Since $x \notin A^*$, it follows that $\psi^t x \in B_{\e/2}(A)$ for some $t \geq 0$. Fixing such a $t$, we have for all $n$ sufficiently large that $d(\psi^t x, \psi^{t - t_n} x_n) < \e/2$ by continuity, hence $\psi^{t - t_n} x_n \in B_\e(A)$ holds for all such $n$. This is in contradiction to the fact that $V_\e$ is a prerepeller on taking $n$ is large enough so that $t_n - t \geq T$, where $T$ is as in the definition of prerepeller (Definition \ref{defn:prerepeller}). We conclude that $R_\e \subset A^*$.
}
%Then there is a sequence of times $t_n' \to \infty$ for which $\psi^{t_n'} x \to x^* \in A$ as $n \to \infty$. In particular, $\psi^{t_n'} x \in U$ for all $n$ sufficiently large, and so from Lemma \ref{lem:asympCompact2} it follows that $\psi^t x \in B_{\e/2}(A)$ for all $t$ sufficiently large. For each $n$, fix an element $x_n \in (\psi^{t_n})^{-1} V_\e$ for some $t_n \geq n$ for which $x_n \in (\psi^n)^{-1}B_{\e/2}(\psi^n x)$. Then $\psi^{n} x_n \in B_\e(A)$, which is a contradiction. Thus $x \in A^*$ must hold.

For the other inclusion, let $x \in A^*$ and note that $\psi^t x \notin B_{2\e}(A)$ for all $t$ sufficiently large (since otherwise $\omega(x) \cap A \neq \emptyset$ by asymptotic compactness). Thus $x \in (\psi^t)^{-1} V_\e$ for all large $t$, and so $x \in R_\e$ follows.
\end{proof}

Although we do not assume invertibility of the time-$t$ maps, we do occasionally need to refer to negative trajectories when they do exist. 
\begin{defn}
Let $x \in X$; we say that $x$ admits a \emph{negative continuation} if $\psi^t x$ exists for all $t \leq 0$.
%if there is a mapping $\sigma : \R \to X$ for which $\sigma(0) = x$ and $\psi^t(\sigma(s)) = \sigma(t + s)$ for any $s \in \R, t \geq 0$. The mapping $\sigma$ is referred to as the negative continuation of $x$.
\end{defn}
\noindent By injectivity of the time-$t$ maps, a negative continuation is unique if it exists. %Abusing notation, we will use the notation $\{\psi^t x\}_{t \in \R}$ when referring to negative continuations in the future.

When $x \in X$ has a negative continuation, we write $\omega^*(x)$ for the backwards limit set of $\{\psi^t x\}_{t \leq 0}$. The following is a consequence of Lemma \ref{lem:repellersRepel}.

\begin{lem}
Let $x \in X \setminus A$, and assume that $x$ has a negative continuation. Then, $\omega^*(x) \subset A^*$.
\end{lem}
\begin{proof}
Note that $\omega^*(x)$ may be empty. If it is not, then apply Lemma \ref{lem:repellersRepel} to $\e = \frac12 \dist(x, A)$ and observe that $\{ \psi^t x\}_{t \leq -T} \subset V_\e$, where $T$ is as in Definition \ref{defn:prerepeller} for $V = V_\e$. Consequently any limit point of $\{ \psi^{-t} x\}$ belongs to $R_\e$, which coincides with $A^*$ by Lemma \ref{lem:repellersRepel}.
\end{proof}

%%Under \eqref{eq:semiflowInj}, note that negative continuations are unique when they do exist. 
%When $x \in X$ has a negative continuation $\sigma$, we write $\omega^*(\sigma)$ for the set of accumulation points of $\{\sigma(t) : t \leq 0\}$.
%
%
%
%
%
%\begin{lem}\label{lem:dual2} Let $(A, U)$ be an asymptotically compact attractor pair with dual repeller pair $(A^*, V)$ as in Lemma \ref{lem:dual}.
%\begin{itemize}
%\item[(a)] For any $x \in X \setminus A^*$, we have that $\omega(x)$ is nonempty and $\subset A$.
%\item[(b)] Let $x \in X \setminus A$ have a negative continuation $\sigma$. Then $\omega^*(\sigma) \subset A^*$.
%\end{itemize}
%\end{lem}
%\begin{proof}
%Part (a) is immediate from the definition of $A^*$. For (b), let $t_n \geq 0, t_n \to \infty$ be such that $\psi^{-t_n} v \to v^*$ for some $v^* \in X$. Assume for the sake of contradiction that $v^* \notin A^*$. Then, $\omega(v^*)$ is nonempty and $\subset A$; in particular $\psi^t v^* \in U$ for $t$ sufficiently large. Note now that $\{\psi^{- t_n} v\}_{n \geq N} \subset U$ for $N$ sufficiently large, and so because $v = \psi^{t_n} \circ \psi^{- t_n} v $, it follows that $v \in \omega(U) = A$. This is a contradiction to the assumption that $v \notin A$, and so we conclude that $\omega^*(v) \subset A^*$.
%
%Note that $\omega^*(v)$ may be empty, even when $v$ has negative continuation.
%\end{proof}

\subsection{Chains, chain recurrence and attactors}\label{subsec:chains}

We complete this section with a brief review of chains and chain recurrence.

Let $\e, T > 0$. For $x,y \in X$ we say that there is an $(\e, T)$-chain from $x$ to $y$ if there is a sequence $x_1, x_2, \cdots, x_{n} \in X$ and times $T_0, T_2, \cdots, T_n \in [T, \infty)$ such that, on setting $x_0 = x, x_{n+1} = y$, we have that $d_X(\psi^{T_i} x_i, x_{i + 1}) < \e$ for all $0 \leq i \leq n$. For a subset $Y \subset X$, we define
\[
\Omega(Y; \e, T) = \{x \in X : \text{ there exists an $(\e, T)$-chain from $y$ to $x$ for some $y \in Y$}\}
\]
and 
\[
\Omega(Y) = \bigcap_{\e >0, T > 0} \Omega(Y; \e, T) \, .
\]

%We say that a point $x \in X$ is \emph{chain recurrent} when $x \in \Omega(\{x\})$; a subset $Y \subset X$ is called \emph{chain transitive} if for any $y \in Y$, we have that $\Omega(\{y\}) = Y$.

\begin{lem}\label{lem:chain1}
Let $Y \subset X$ be a subset for which $Y \subset U$, where $(A, U)$ is an asymptotically compact attractor pair. Then $\Omega(Y) \subset A$. 
\end{lem}
\begin{proof}
We will show the following: for any $\e > 0$ and for $T$ sufficiently large (in terms of $\e$), $\Omega(Y; \e, T) \subset B_{2\e}(A)$. To see this, fix $\e > 0$ and let $T^*$ be sufficiently large so that $\overline{\psi([T, \infty) \times U)} \subset B_\e(A)$ for all $T \geq T^*$ (Lemma \ref{lem:asympCompact2}). Let now $x \in U$ be arbitrary-- it now follows that any finite $(\e, T)$ chain initiated at $x$ will terminate in a point $y \in B_{2 \e}(A)$. In particular, $\Omega(Y; \e, T) \subset B_{2 \e}(A)$ for any $Y \subset U$. This completes the proof.
\end{proof}

\begin{rmk}
In \cite{hurley1991chain, hurley1992noncompact, hurley1995chain, choi2002chain}, a more general definition of chain is used. This broadened definition was designed for use in the non-locally-compact setting, and gives rise to equivalent notions of chain recurrence and chain transitivity in the compact setting. We use the `classical' definition here because we only ever consider the chain transitivity of compact subsets.
\end{rmk}

%\newpage

\section{Banach space preliminaries}\label{sec:banachPrelim}

Here we recall some technical preliminaries on Banach space geometry, in particular the `local' Banach space geometry of finite dimensional and finite codimensional subspaces. 
%All material in \S \ref{subsec:grassmanian}, \ref{subsec:volumesGelfand} is review-- results are either standard or are proved in \cite{Blumenthal20162377, Blumenthal2016}. In \S \ref{subsec:exponentialSplittings} we recall 
%
%The rest of \S \ref{sec:banachPrelim} is standard in finite dimensions 

\medskip

\noindent {\bf Notation.} Throughout this section, $\Bc$ is a Banach space with norm $|\cdot|$.  The Grassmanian $\Gc(\Bc)$ is defined to be the set of nontrivial closed subspaces of $\Bc$. When $E, F \in \Gc(\Bc)$ and $\Bc = E \oplus F$ is a splitting, we write $\pi_{E \ds F}$ for the projection onto $E$ parallel to $F$ (i.e., $F = \ker (\pi_{E \ds F}), E = \operatorname{Range} (\pi_{E \ds F})$). We say that $E, F$ are complements in $\Bc$.

Note that $\pi_{E \ds F}$ is always a bounded linear operator when $E, F$ are closed and $\Bc = E \oplus F$ (by the Closed Graph Theorem).

\subsection{Grassmanian of closed subspaces}\label{subsec:grassmanian}

The Grassmanian $\Gc(\Bc)$ is endowed with a metric, the Hausdorff distance $d_H$, which for $E, F \in \Gc(\Bc)$ is defined by
\[
d_H(E, F) = \max\big\{ \sup_{e \in S_E} d(e, S_F),  \sup_{f \in S_F} d(f, S_E) \big\} \, ;
\]
here we have written $S_E = \{e \in E : |e| = 1\}$ and analogously for $S_F$.

\medskip

For $d \in \N$, write $\Gc_d(\Bc)$ for the subset of $d$-dimensional subspaces, and $\Gc^d(\Bc)$ for the subset of closed $d$-codimensional subspaces.

\begin{lem}[\cite{kato2013perturbation}]
The metric $d_H$ is a complete metric for $\Gc(\Bc)$. The subsets $\Gc_d(\Bc)$ and $\Gc^d(\Bc)$ are closed in $(\Gc(\Bc), d_H)$ for any $d \in \N$.
\end{lem}

\medskip

For computations it is simpler to work with the \emph{gap} between subspaces, defined by
\[
\Gap(E, F) = \sup_{e \in S_E} d(e, F) \, ;
\]
then, 
\begin{align}\label{eq:gapEquiv}
\frac12 d_H(E, F) \leq \max\{\Gap(E, F), \Gap(F, E)\} \leq d_H(E, F) \, .
\end{align}
For proof, see \cite{kato2013perturbation}.

The following Lemma makes computations involving $\Gap$ simpler when one works with finite dimensional or codimensional subspaces.

\begin{lem}[Lemma 2.6 in \cite{Blumenthal20162377}] \label{lem:symmCloseness} Let $d \in \N$.
\begin{itemize}
\item[(a)] Let $E, E' \in \Gc_d(\Bc)$. Then \[ \Gap(E', E) \leq \frac{d \Gap(E, E')}{ 1- d \Gap(E, E')}\] whenever the denominator of the right-hand side is positive.
\item[(b)] Let $F, F' \in \Gc^d(\Bc)$. Then \[ \Gap(F', F) \leq \frac{d \Gap(F, F')}{ 1- d \Gap(F, F')}\] whenever the denominator of the right-hand side is positive.
\end{itemize}
\end{lem}

\subsubsection{Complementation in $\Gc_d(\Bc), \Gc^d(\Bc)$; angles between subspaces}

Not every closed subspace of a Banach space possesses a closed complement. However, for finite dimensional and closed finite codimenisonal subspaces, we have the following.
\begin{lem}[III.B.10 and III.B.11 in \cite{wojtaszczyk1996banach}] \label{lem:compExist} Let $d \in \N$.
\begin{itemize}
\item For any $E \in \Gc_d(\Bc)$, there exists a subspace $F \in \Gc^d(\Bc)$ complementing $E$ for which $|\pi_{E \ds F}| \leq \sqrt d$.
\item For any $F \in \Gc^d(\Bc)$, there exists $E \in \Gc_d(\Bc)$ complementing $F$ for which $|\pi_{F \ds E}| \leq \sqrt d + 1$.
\end{itemize}
\end{lem}
\noindent Lemma \ref{lem:compExist} can be used to produce `good' bases of finite-dimensional spaces: for any $d \in \N$, there is a constant $C_d > 0$ such that for any $d$-dimensional subspace $E \subset \Bc$, there is a basis $v_1, \cdots, v_d$ of unit vectors for which
\[
N[v_1, \cdots, v_d] := \sum_{i = 1}^d |\pi_{\langle v_i \rangle \ds \langle v_j : j \neq i \rangle }| 
\]
satisfies $N[v_1, \cdots, v_d] \leq C_d$.

\smallskip

It is sometimes useful to consider an analogue of the notion of \emph{angle} between subspaces of a Banach space. The following is a standard construction.

\begin{defn}\label{defn:minAngle}
Let $E, F \in \Gc(\Bc)$. The \emph{minimal angle} $\theta(E, F) \in [0, \pi/2]$ between $E, F$ is defined by
\[
\sin \theta(E, F) = \min\{|e - f| : e \in E, |e| = 1, f \in F\} \, .
\]
\end{defn}

\noindent A quick computation (see, e.g., \cite{Blumenthal20162377}) shows that when $E, F \in \Gc(\Bc)$ are complements, we have that
\begin{align}\label{eq:minAngle}
\sin \theta(E, F) = |\pi_{E \ds F}|^{-1} \, .
\end{align}

Complementation is an open condition.
\begin{lem}\label{lem:compOpen}
Let $E, F \in \Gc(\Bc)$ be complements. Then, $E', F$ are complements for any $E' \in \Gc(\Bc)$ with $d_H(E, E') < \sin \theta(E, F)$. Additionally, we have the estimates
\[
|\pi_{E' \ds F}| \leq {|\pi_{E \ds F}| \over 1 - d_H(E, E') |\pi_{E \ds F}| } \, 
\]
and
\[
|\pi_{F \ds E'}|_E| \leq 2 |\pi_{E' \ds F}| \, d_H (E, E') \, .
\]
\end{lem}
For a proof, see the Appendix of \cite{Blumenthal2016}.

\begin{lem}\label{lem:projectionContinuity}
Let $E, F \in \Gc(\Bc)$ be complements. Then, there are open neighborhoods $\Nc_E, \Nc_F \subset \Gc(\Bc)$ of $E, F$, respectively, such that (i) for any $E' \in \Nc_E, F' \in \Nc_F$, we have that $E', F'$ are complements, and (ii) the map $(E' , F') \mapsto \pi_{E' \ds F'}$ is continuous on $\Nc_E \times \Nc_F$ in the operator norm.
\end{lem}

\begin{proof}
With $E, F$ fixed, set $\Nc_E := \{ E' \in \Gc(\Bc) : d_H(E', E) < \frac12 |\pi_{E \ds F}|^{-1} \}$. For any $E' \in \Nc_E$, note that 
\[
|\pi_{F \ds E'}| \leq 2 |\pi_{E' \ds F}| \leq 4 |\pi_{E \ds F}| \leq 8 |\pi_{F \ds E}| \, ,
\]
and so $\Nc_F = \{ F' \in \Gc(\Bc) : d_H(F', F) < \frac{1}{10} |\pi_{F \ds E}| \}$ together with $\Nc_E$ satisfy item (i) by Lemma \ref{lem:compOpen}.

To prove continuity, let $E_1, E_2 \in \Nc_E, F_1, F_2 \in \Nc_F$. Then
\begin{align*}
\pi_{E_1 \ds F_1} - \pi_{E_2 \ds F_2} &= \big( \pi_{E_1 \ds F_1} - \pi_{E_1 \ds F_2}  \big) + \big( \pi_{E_1 \ds F_2} -  \pi_{E_2 \ds F_2} \big) \\
& = \big( \pi_{F_2 \ds E_1} - \pi_{F_1 \ds E_1} \big) + \big( \pi_{E_1 \ds F_2} -  \pi_{E_2 \ds F_2} \big) \, .
\end{align*}
The norm of the second parenthetical term can be estimated as
\[
(*) = |(\pi_{E_1 \ds F_2} -  \pi_{E_2 \ds F_2}) \circ (\pi_{E_1 \ds F_2} + \pi_{F_2 \ds E_1}) | \leq |\pi_{E_1 \ds F_2}| \cdot |\pi_{F_2 \ds E_2}|_{E_1} | \, .
\]
By Lemma \ref{lem:compOpen}, $|\pi_{E_1 \ds F_2}|$ and $|\pi_{E_2 \ds F_2}|$ are bounded independently of $E_1, E_2, F_2$, and so $(*)$ is bounded $\leq \text{Const. } d_H(E_1, E_2)$. Similar arguments yield the bound $|\pi_{F_2 \ds E_1} - \pi_{F_1 \ds E_1}| \leq \text{Const. } d_H(F_1, F_2)$. This completes the proof of (ii).
\end{proof}

\subsection{Continuous subbundles of a Banach bundle}

In this subsection, we let $(Z, d_Z)$ be a compact metric space and consider the Banach bundle $\mathcal Z = Z \times \Bc$ over $Z$. We sometimes abuse notation and regard $\mathcal Z_z = \{z \} \times \Bc$ as a vector space for $z \in Z$.
\begin{defn}
Let $\Cc \subset \mathcal Z$. We say that $\Cc$ is a continuous subbundle if the following holds: (i) for any $z \in Z$, $\Cc_z = \mathcal Z_z \cap \Cc$ is a closed subspace, and (ii) the assignment $z \mapsto \Cc_z$ is continuous as a map $(Z, d_Z) \rightarrow (\Gc(\Bc), d_H)$. 
\end{defn}

We now give criteria for checking when closed subsets of $\mathcal Z$ are continuous subbundles.

\begin{lem}\label{lem:compactSubbundle}
Let $\Cc \subset \mathcal Z$ be a closed subset for which $\Cc_z = \Cc \cap \{z \} \times \Bc$ is a finite dimensional subspace of finite dimension $d$ independent of $z$. Assume that the unit sphere $S_{\Cc} = \{(z, v) \in \Cc : |v| = 1\}$ of $\Cc$ is compact. Then, $\Cc$ is a continuous subbundle of $\mathcal Z$.
\end{lem}
\begin{proof}[Proof of Lemma \ref{lem:compactSubbundle}]
Let $z_n \to z$ be a convergent sequence in $Z$. We will show that $C_{z_n} \to C_z$ in the Hausdorff distance $d_H$. It suffices to find a subsequence $\{n_i\}$ for which $d_H(C_{z_{n_i}}, C_z) \to 0$.

Let us fix some notation. For each $n$, let $v_n^1, \cdots, v_n^d$ denote a basis of $\Cc_{z_n}$ of unit vectors for which $N[v_n^1, \cdots, v_n^d] \leq C_d$, where $C_d$ depends only on $d$ (Lemma \ref{lem:compExist}).

Using the compactness of $S_{\Cc}$, we can pass to a subsequence $n_l$ along which $v_{n_l}^j$ converges to a unit vector $\hat v^j \in \Cc_z$ for each $1 \leq j \leq d$. This implies 
$\pi_{\langle v^i_{n_l} \rangle \ds \langle v_{n_l}^j : j \neq i \rangle } \to \pi_{\langle \hat v^i \rangle \ds \langle \hat v^j : j \neq i \rangle}$ in the operator norm (use, e.g., Lemma \ref{lem:projectionContinuity}). Since $N[v_n^1, \cdots, v_n^d] \leq C_d$ for all $n$, we conclude that the cluster point $\{ \hat v^1, \cdots, \hat v^d \} \in \Cc_z$ is a linearly independent set, hence a basis for $\Cc_z$. It is now simple to check that $d_H(\Cc_{z_{n_l}}, \Cc_z) \to 0$.
\end{proof}

%\begin{lem}\label{lem:closedSubbundle2}
%Assume that $\Bc$ is separable. Let $\Cc \subset \Zc$ be a continuous finite dimensional subbundle. Assume that $\Dc \subset \Zc$ is a closed subset such that for each $z \in Z$, (i) $\Dc_z = \Dc \cap \Zc_z$ is a closed subspace of $\Zc_z$, and (ii) that $\Dc_z$ complements $\Cc_z$ in $\Bc$, where $\sin \theta(\Dc_z, \Cc_z) \geq c$ for a constant $c >0 $ independent of $z \in Z$. Then $\Dc$ is a continuous subbundle of $\Zc$.
%\end{lem}

\noindent We note that closed subsets of $\Zc$ with finite-dimensional fibers need not be compact, nor continuous subbundles:
\begin{ex}
Let $Z = \{1, \frac12, \frac13, \cdots\} \cup \{0\}$ with the usual metric, and let $\Bc = \ell^2(\N)$ with standard basis $e_1, e_2, \cdots$. Define $\Cc_{1/n} = \langle e_n \rangle$ and $\Cc_0 = \langle e_1 \rangle$. Then $\Cc$ is closed (albeit noncompact), has one-dimensional fibers, and yet is not a continuous subbundle.
\end{ex}

\subsection{Projectivization}

Let $\P \Bc$ denote the projective space of $\Bc$. Specifically, we define the equivalence relation $\sim$ on $\Bc \setminus \{0 \}$ by setting $v \sim w$ iff $v = \lambda w$ for some $\lambda \in \R \setminus \{0\}$; we write $\P v \in \P \Bc$ for the representative of $v$. For $v, w \in \Bc \setminus \{0\}$, we define the projective metric
\begin{align}\label{eq:projectiveMetric}
d_\P (\P v, \P w) = \min \bigg\{ \bigg|  \frac{v}{|v|_b} - \frac{w}{|w|_b} \bigg|  , \bigg| \frac{v}{|v|_b} + \frac{w}{|w|_b} \bigg| \bigg\} \, ,
\end{align}
where for $v \in \Bc \setminus \{0\}$ we write $\P v \in \P \Bc$ for the equivalence class of $v$. 

\medskip

The following estimate is frequently useful.

\begin{lem}\label{lem:projectiveEstimate}
Let $E \subset \Gc(\Bc)$ be a complemented subspace with complement $F \in \Gc(\Bc)$, and let $v \in \Bc \setminus \{0\}$ be a unit vector. Write $\P E = \{ \P e : e \in E \setminus \{0\}\}$. Then
\[
\frac{|\pi_{F \ds E} v|}{|\pi_{F \ds E}|} \leq d_\P (\P v, \P E) \leq 2 |\pi_{F \ds E} v| \, .
\]
Here, $d_\P(\P v, \P E) = \inf\{d_\P (\P v, \P e) : \P e \in \P E\}$.
\end{lem}
\begin{proof}
For the first inequality, fix $\a > 1$ and let $e \in E$ be a unit vector for which $|v - e| \leq \a d_\P(\P v, \P E)$. Then 
\[
|\pi_{F \ds E} v| = |\pi_{F \ds E} (v - e) | \leq |\pi_{F \ds E}| \cdot |v - e| \leq \a |\pi_{F \ds E}| \cdot d_{\P}(\P v, \P E) \, ,
\]
and so the desired inequality obtains on taking $\a \to 1$.

For the second inequality, let $e = \pi_{E \ds F} v, f = \pi_{F \ds E} v = v - e$, and note that 
\begin{align*}
d_\P(\P v, \P E) & \leq |v - |e|^{-1} e| = \big|(1 - |e|^{-1}) e + f \big| \leq |f| + \big| |e| - 1 \big| \\
& = |f| + \big| |e| - |v| \big| \leq |f| + |v - e| = 2 |f| = 2 |\pi_{F \ds E} v| \, .
\end{align*}
\end{proof}

%\begin{lem}\label{lem:projCont}
%The projectivized semiflow $\P \Phi$ is continuous as a map $[0,\infty) \times \P \Vc \to \P \Vc$. 
%\end{lem}
%\begin{proof}
%That $\Phi : [0,\infty) \times \Vc \to \Vc$ is continuous follows by the strong continuity of $(t, b) \mapsto \Phi^t_b$. To check continuity of $\P \Phi$, observe that
%
%\end{proof}

%\begin{lem}\label{lem:continuousSemiflowProjective}
%Let $\Phi$ be an injective linear semiflow of bounded operators as in \S \ref{sec:introResults}. Then, the projectivized semiflow $\P \Phi$, defined by setting $\P \Phi( t, \P v) := \P(\Phi^t(v))$ is continuous as a map $\P \Phi : [0,\infty) \times  \P \Vc  \to \P \Vc$.
%\end{lem}
%
%\begin{proof}
%
%\[
%\Phi^t_b v - \Phi^{t_n}_{b_n} v_n = \Phi^t_b v - \Phi^t_b v_n + \Phi^t_b v_n - \Phi^{t_n}_b v_n + \Phi^{t_n}_b v_n - \Phi^{t_n}_{b_n} v_n
%\]
%
%
%{ Proof of Lemma on projective space continuity}
%\end{proof}

\subsection{Induced volumes, determinants and Gelfand numbers}\label{subsec:volumesGelfand}

\begin{defn}
Let $E \subset \Bc$ be a finite-dimensional subspace. We write $m_E$ for the \emph{induced volume} on $E$, which is defined to be the Haar measure on $E$ normalized so that
\[
m_E\{v \in E : |v| \leq 1\} = \omega_{\dim E} \, .
\]
Here, $\omega_q$ denotes the volume of the $q$-dimensional Euclidean unit ball in $\R^q$.
\end{defn}

Determinants on finite dimensional subspaces can now be defined as volume ratios: given a linear operator $T : \Bc \to \Bc$ and a finite dimensional subspace $E \subset \Bc$, we define
\[
\det(T | E) = \begin{cases} \frac{m_{T E}(T B)}{m_E(B)} & T|_E \text{ is injective,} \\ 0 & \text{else.} \end{cases}
\]
Here $B \subset E$ is any Borel set with positive $m_E$ measure; that $\det(T|E)$ does not depend on the particular choice of $B$ follows from the uniqueness of Haar measure up to scaling.

\begin{lem}\label{lem:determinantAngleEst}
Let $E, F \subset \Bc$ be finite-dimensional subspaces, $\dim E = k, \dim F = l$, and let $T : \Bc \to \Bc$ be a bounded linear operator such that $T|_{E \oplus F}$ is injective. Write $E' = T E, F' = T F$. Then,
\[
C^{-1} \big( \sin \theta(E', F') \big)^k \leq \frac{\det(T | E \oplus F)}{\det (T | E) \det (T | F)} \leq C \big( \sin \theta(E, F) \big)^k \, ,
\]
where $C$ is a constant depending only on $q = k + l$.
\end{lem}

\begin{defn}
Let $q \in \N$. For a linear operator $T : \Bc \to \Bc$, the maximal $q$-dimensional volume growth $V_q(T)$ is defined by
\[
V_q(T) = \sup\{\det(T | E) : E \subset \Bc, \, \dim E = q \} \, .
\]
\end{defn}

For bounded linear operators $T$ of a Hilbert space, the quantity $V_q(T)$ is given by the product $\prod_{i = 1}^q \sigma_i(T)$, where $\sigma_i(T)$ denotes the $i$-th
singular value of $T$ (that is, the $i$-th eigenvalue, counted in descending order, of the positive semi-definite self-adjoint operator $T^* T$). 

For operators of a Banach space, there is no `canonical' definition of singular value. Instead one often works with one of a variety of surrogate notions, 
called $s$-numbers in the literature-- see, e.g., \cite{pietsch1986eigenvalues}. The following $s$-number is useful for our purposes.

\begin{defn}
Let $T : V \to V'$ be a bounded linear operator of Banach spaces $(V, |\cdot|), (V', |\cdot|')$. For $k \geq 1$, the $k$-th Gelfand number $c_k(T)$ is defined to be
\[
c_k(T) = \inf \{|T|_F| : F \subset V, \, \codim F = k - 1\} \, .
\]
\end{defn}
For bounded linear operators on Hilbert spaces, the Gelfand numbers coincide with singular values, hence $V_q(T) = \prod_{i = 1}^q c_i(T)$. In the Banach space setting, we can recover the following weaker relation.
\begin{lem}\label{lem:equivVolumeGelfand}
For each $q \in \N$ there is a constant $C \geq 1$, depending only on $q$, with the following property. For any bounded linear $T : \Bc \to \Bc$, we have that
\[
C^{-1} V_q(T) \leq \prod_{i = 1}^q c_i(T) \leq C V_q(T) \, .
\]
\end{lem}

\subsection{Exponential separations for Banach space cocycles}\label{subsec:exponentialSplittings}

Here we recall the definition of exponential separation and several related results we will need later on. Throughout \ref{subsec:exponentialSplittings}, $\Phi$ is a linear semiflow on $\Vc = B \times \Bc$ satisfying (H1) -- (H4) as in \S \ref{sec:introResults}. We note that Lemma \ref{lem:expSepUnique} and Proposition \ref{prop:moveDownSubspace} are used heavily in \S \ref{sec:cleanUp}.

\begin{defn}
Let $\Vc = \Ec \oplus \Fc$ be a Whitney splitting of $\Vc$ into continuously varying, forward invariant subbundles for which $\dim \Ec < \infty$. We say that $\Ec, \Fc$ are \emph{exponentially separated} if there exist constants $K , \gamma > 0$ with the following property: for any $t > 0$, we have that
\[
|\Phi^t_b|_{\Fc_b}| \leq K e^{- \gamma t} m(\Phi^t_b|_{\Ec_b}) \, .
\]
Here, for a linear operator $T$ on $\Bc$ and a subspace $E \subset \Bc$ we write $m(T|_E) = \inf \{ |T e| : e \in E, |e| =1 \}$ for the minimum norm of $T|_{E}$.
\end{defn}

Note that by injectivity and finite-dimensionality of $\Ec$, it holds automatically that $\Phi_b^t : \Ec_b \to \Ec_{\phi^t b}$ is an isomorphism for any $b \in B, t \geq 0$. In particular, all points of $\Ec$ possess negative continuation and $\Ec$ is backwards invariant.

\begin{defn}
We say that $\Phi$ has an \emph{exponential splitting of index $k$} if there is an exponential splitting $\Vc = \Ec \oplus \Fc$ for $\Phi$ for which $\dim \Ec = k$.
\end{defn}

{

\begin{lem}\label{lem:expSepUnique}
Let $k \in \N$. If $\Phi$ has an exponential splitting of index $k$ and $\Vc = \Ec \oplus \Fc = \Ec' \oplus \Fc'$ are two exponential splittings for $\Phi$ for which $\dim \Ec  = \dim \Ec' = k$, then $\Ec = \Ec'$ and $\Fc = \Fc'$.
\end{lem}
\begin{proof}
Let $\Vc = \Ec \oplus \Fc = \Ec' \oplus \Fc'$ be two exponential splittings for $\Phi$ for which $\dim \Ec = \dim \Ec'$. Let $K, \gamma > 0$ be such that
\begin{align}\label{eq:expSplitUniqueEstimate}
|\Phi^t_{\hat b}|_{\Fc_{\hat b}^{(')}}| \leq K e^{- \gamma t} m(\Phi^t_{\hat b}|_{\Ec^{(')}_{\hat b}})
\end{align}
for all $\hat b \in B, t \geq 0$.

We first show the following.
{ \begin{cla}\label{cla:complementUniqueness}
For any $b \in B$, we have that $\Ec_b' \cap \Fc_b = \{0\}$, hence $\Vc_b = \Ec_b' \oplus \Fc_b$.
\end{cla}
\noindent The Claim implies
\begin{align}\label{eq:anglePositiveLowerBound}
\inf_{b \in B} \sin \theta(\Ec_b', \Fc_{b}) > 0 \, .
\end{align}
To deduce \eqref{eq:anglePositiveLowerBound} from Claim \ref{cla:complementUniqueness}, observe that $b \mapsto \pi_{\Ec'_b \ds \Fc_b}$ is continuous in the operator norm (Lemma \ref{lem:projectionContinuity}), and so $\sup_{b \in B} |\pi_{\Ec'_b \ds \Fc_b}| = (\inf_{b \in B} \sin \theta(\Ec_b' , \Fc_b) )^{-1} < \infty$.
}
\begin{proof}[Proof of Claim]
For the sake of contradiction, assume that $\Ec_b' \cap \Fc_b \neq \{0\}$ for some $b \in B$. Without loss we may assume $\Fc_b\setminus \Fc_b' \neq \emptyset$, since otherwise $\Fc_b = \Fc_b'$. It follows that there is some unit vector $f' \in \Fc'_b$ for which $e = \pi_{\Ec_b \ds \Fc_b} f' \neq 0$. Write $f = f' - e = \pi_{\Fc_b \ds \Ec_b} f'$.

 Let now $\hat e \in \Ec_b' \cap \Fc_b$ be a unit vector. Since $\hat e \in \Ec_b', f' \in \Fc_b'$, we have $|\Phi^t_b f'| \leq K e^{- \gamma t} |\Phi^t_b \hat e|$. Using $\hat e \in \Fc_b$, we now estimate
 \[
 K e^{- \gamma t} |\Phi^t_b |_{\Fc_b}| \geq K e^{- \gamma t} |\Phi^t_b \hat e| \geq |\Phi^t_b f'| \geq |\Phi^t_b e| - |\Phi^t_b f| \geq m(\Phi^t_b|_{\Ec_b}) \cdot|e| - |\Phi^t_b|_{\Fc_b}| \cdot |f|\, .
 \]
Rearranging, one obtains that the ratio $m(\Phi^t_b|_{\Ec_b}) / |\Phi^t_b|_{\Fc_b}|$ is bounded by a constant independent of time-- this contradicts the exponential separation of $\Ec, \Fc$, hence a contradiction.
\end{proof}

Let us now return to the proof of Lemma \ref{lem:expSepUnique}.

\noindent {\bf Proving $\Ec = \Ec'$. } Let $b \in B$ and $e' \in \Ec'$ be a unit vector, decomposed as $e' = e + f$ according to the splitting $\Vc_b = \Ec_b \oplus \Fc_b$. We will show $f = 0$, hence $\Ec_b' \subset \Ec_b$ for all $b$; equality follows on recalling that $\dim \Ec = \dim \Ec'$ by assumption.

For each $t > 0$, let $e_{- t}'$ be such that $\Phi^t_{\phi^{-t} b} e_{- t}' = e'$, and write $e'_{- t} = e_{-t} + f_{-t}$ according to the splitting $\Vc_{\phi^{-t} b} = \Ec_{\phi^{-t} b} \oplus \Fc_{\phi^{-t} b}$. Note that by equicontinuity of $\Ec, \Fc$, we have that $\Phi^t_{\phi^{-t} b} e_{-t} = e, \Phi^t_{\phi^{-t} b} f_{-t} = f$. 

To begin, observe that 
\[
|f| \leq |\Phi^t_{\phi^{-t} b}|_{\Fc_{\phi^{-t} b}}| \cdot |f_{-t} | \leq C'  |\Phi^t_{\phi^{-t} b}|_{\Fc_{\phi^{-t} b}}|  \cdot |e_{-t}'| \, ,
\]
where $C' = \sup_{\hat b} |\pi_{\Fc_{\hat b} \ds \Ec_{\hat b} }|$. We now estimate $|e_{-t}'|$:
\[
1 = |e'| = |\Phi^t_{\phi^{-t} b} e_{-t}'| \geq |\Phi^t_{\phi^{-t} b} e_{-t}| - |\Phi^t_{\phi^{-t} b} f_{-t}| \geq m(\Phi^t_{\phi^{-t} b}|_{\Ec_{\phi^{-t} b}}) |e_{-t}| - |\Phi^t_{\phi^{-t} b}|_{\Fc_{\phi^{-t} b}}| \cdot |f_{-t}|
\]
From \eqref{eq:anglePositiveLowerBound}, we have that $d(\hat e', \Fc_{ b}) \geq c := \inf_{\hat b \in B} \sin \theta(\Ec_{\hat b}', \Fc_{\hat b}) > 0$ for all $ \hat e' \in \Ec_{\hat b}', |\hat e'| = 1$. In particular, $|\pi_{\Ec_{\hat b} \ds \Fc_{\hat b}} \hat e'| = |\hat e' - \pi_{\Fc_{\hat b} \ds \Ec_{\hat b}} \hat e'| \geq d(\hat e', \Fc_{\hat b}) \geq c$. Applying to $\hat e' = e_{-t}' / |e_{-t}'|$, we obtain that $|e_{-t}| \geq c |e_{-t} '|$. In conjunction with the estimate $|f_{-t}| \leq C' |e_{-t}'|$, we conclude that
\begin{gather*}
|e'_{-t}| \leq \big( c m(\Phi^t_{\phi^{-t} b}|_{\Ec_{\phi^{-t} b}}) - C' |\Phi^t_{\phi^{-t} b}|_{\Fc_{\phi^{-t} b}}| \big)^{-1} \, ,  \quad \text{hence} \\
|f| \leq \frac{C'  |\Phi^t_{\phi^{-t} b}|_{\Fc_{\phi^{-t} b}}|}{c m(\Phi^t_{\phi^{-t} b}|_{\Ec_{\phi^{-t} b}}) - C' |\Phi^t_{\phi^{-t} b}|_{\Fc_{\phi^{-t} b}}|} \, .
\end{gather*}
Applying \eqref{eq:expSplitUniqueEstimate} and taking $t \to \infty$, we conclude that $f =0$, as desired.

\medskip

\noindent {\bf Proving $\Fc = \Fc'$. } As before, it suffices to check $\Fc \subseteq \Fc'$. For the sake of contradiction, let $f \in \Fc_b, b \in B$ be such that $f = e' + f'$ according to the splitting $\Ec_b' \oplus \Fc_b'$ and assume $e' \neq 0$. Writing $f_t = \Phi^t_b f, e_t' = \Phi^t_b e', f_t' = \Phi^t_b f'$, observe that
\[
d\bigg(\frac{f_t}{|f_t|}, \Ec_{\phi^t b}'\bigg) \leq \bigg| \pi_{\Fc_{\phi^t b}' \ds \Ec_{\phi^t b}'} \frac{f_t}{|f_t|} \bigg| = \frac{|f_t'|}{|f_t|} \leq \frac{|f_t'|}{|e_t'| - |f_t'|} \, .
\]
The right-hand ratio goes to zero by \eqref{eq:expSplitUniqueEstimate} since $e' \neq 0$, and so we obtain that $\inf_{b \in B} d(\Fc_b, \Ec_b') = 0$. By compactness, the infimum is attained-- this contradicts Claim \ref{cla:complementUniqueness}, however, and so we conclude $e' = 0$. Thus we have shown $f \in \Fc_b'$, as desired.
\end{proof}

}

The following is a characterization of exponential separation in terms of exponential growth rates of Gelfand numbers-- it generalizes a similar criterion developed by Bochi and Gourmelon for finite-dimensional linear cocycles \cite{bochi2009some}.

\begin{thm}[\cite{blumenthal2015characterization}]\label{thm:characterizeExpSep}
The following are equivalent.
\begin{itemize}
\item $\Phi$ has an exponential splitting of index $k$ for some $k \in \N$.
\item The inequality
\[
\sup_{\e \in [0,1]} c_{k + 1}(\Phi^t_{\phi^\e b}) \leq K e^{- \gamma t} c_k(\Phi^{t + 1}_b)
\]
holds for all $t \geq 0, b \in B$, where $K, \gamma > 0$ are constants.
\end{itemize}
Moreover, the exponential splitting $\Vc = \Ec \oplus \Fc$ of index $k$ satisfies
\begin{align}\label{eq:volumeAchieved}
\hat C^{-1} V_k(\Phi^t_b) \leq \det(\Phi^t_b | \Ec_b) \leq \hat C  V_k(\Phi^t_b)
\end{align}
and 
\begin{align}\label{eq:gelfandAchieved}
\hat C^{-1} c_{k + 1}(\Phi^t_b) \leq |\Phi^t_b |_{\Fc_b}| \leq \hat C c_{k + 1}(\Phi^t_b) \, .
\end{align}
for all $t \geq 0, b \in B$, where $\hat C \geq 1$ is a constant.
\end{thm}

Lastly, we record the following consequence of Theorem \ref{thm:characterizeExpSep}, which will be used in \S \ref{sec:cleanUp} as part of an inductive procedure.

%To prove Lemma \ref{lem:finestDecomp}, it suffices to prove the following.
\begin{prop}\label{prop:moveDownSubspace}
Let $\Vc = \Ec \oplus \Fc$ be any exponentially separated splitting, and let $k > \dim \Ec$. Then $\Phi$ has an exponential splitting of index $k$ if and only if $\Phi|_\Fc$ has an exponential splitting of index $k - \dim \Ec$.
%Let $\Ac = \P \Ec$ be an asymptotically compact attractor and let $1 \leq n < N + 1$ be such that $\dim \Ec > \dim \Ec_. Then $\P (\Ec \cap \Fc_n)$ is an asymptotically compact attractor for $\Phi^t|_{\Fc_n}$.
\end{prop}

\begin{proof}
By Theorem \ref{thm:characterizeExpSep}, it suffices to establish the following. Let $\Vc = \Ec \oplus \Fc$ be an exponential splitting and let $k = \dim \Ec$. Then, for every $l \geq 1$ there is a constant $\hat C_l$ such that for any $b \in B, t \geq 0$, we have
\[
c_{l + k} (\Phi^t_b) \leq  c_l(\Phi^t_b|_{\Fc_b}) \leq \hat C_l c_{l + k} (\Phi^t_b) \, .
\]

To start, observe that
\[
c_l(\Phi^t_b|_{\Fc_b}) = \inf\{ |\Phi^t_b|_F| : F \subset \Fc_b, \codim F = l + k - 1 \} \geq \inf \{|\Phi^t_b|_F| : \codim F = l + k - 1\} = c_{l + k}(\Phi^t_b)
\]
for every $l \geq 1$. Thus it suffices to prove the upper bound on $c_l(\Phi^t_b|_{\Fc_b})$.

%From Lemma \ref{lem:equivVolumeGelfand}, we have
Let $\hat F \subset \Fc_b$ be a $l$-dimensional subspace for which $\det(\Phi^t_b | \hat F) \geq \frac12 V_l(\Phi^t_b|_{\Fc_b})$. Using Lemma \ref{lem:determinantAngleEst}, we estimate
\begin{align*}
V_{l + k}(\Phi^t_b) & \geq \det(\Phi^t_b |\Ec_b  \oplus \hat F  ) \geq C^{-1} |\pi_{\Ec_{\phi^t b} \ds \Fc_{\phi^t b}}|^{-k} \det(\Phi^t_b | \Ec_b) \cdot \det(\Phi^t_b | \hat F) \\ 
& \geq C^{-1} V_k(\Phi^t_b) \cdot V_l(\Phi^t_b |_{\Fc_b}) \, ,
\end{align*}
where $C > 0$ is a generic constant independent of $b, t$. In the last line we have used \eqref{eq:volumeAchieved} and that $\sup_b |\pi_{\Ec_b \ds \Fc_b}| < \infty$.

We now apply Lemma \ref{lem:equivVolumeGelfand} to the left and right hand sides, obtaining
\[
C \cdot \prod_{i = 1}^{l + k} c_i(\Phi^t_b) \geq \prod_{i = 1}^k c_i(\Phi^t_b) \cdot \prod_{i = 1}^l c_i(\Phi^t_b|_{\Fc_b}) \geq \bigg( \prod_{i = 1}^{k + l - 1} c_i(\Phi^t_b) \bigg)  \cdot c_{k + l} (\Phi^t_b|_{\Fc_b}) 
\]
on applying the lower bound on $c_{l'}(\Phi^t_b)$ for $1 \leq l' < l$. On canceling out we conclude the desired upper bound on $c_l(\Phi^t_b|_{\Fc_b})$.

\end{proof}

\section{Asymptotically compact attractors and splittings}\label{sec:hard}

%- Theorem A: Let B be chain transitive. Then, there is a possibly countable collection of finite-dimensional subbundles E_1, E_2, \cdots, for which (A) each \P E_i is chain transitive for \P \Phi and (B) any asymptotically compact attractor A for \P \Phi coincides with \P (E_1 + \cdots + E_n) for some n.

%    - Theorem A should be thought of as a generalization of Selgrade?s theorem to this setting.
%    - Theorem A does not capture all chain recurrent points for the projectivization. This is a more subtle question we do not address here.

%- Definition: Exponentially separated subbundles.

%- Theorem B: Let B be chain transitive. Then, exponentially separated subbundles of finite dimension are in one-to-one correspondence with asymptotically compact attractors of the projectivized flow \P \Phi.

%    - Care must be taken in the formulation of Theorem B: easy to come up with cases when an ?attractor? is not associated w an exponential splitting.

%Plan: in the next subsection we prove the 'main prop'.  What follows is the deduction of Theorem B, then Theorem A.

Our goal in \S \ref{sec:hard} is to prove the following `main' proposition.

\begin{prop}\label{prop:compSubspace}
Let $\phi$ be a chain-transitive flow on a compact metric space $B$, $\Bc$ a separable Banach space, and $\Phi$ a linear semiflow on $\Vc = B \times \Bc$ satisfying (H1) -- (H4) in \S \ref{sec:introResults}. Let $\Ac$ be an asymptotically compact attractor for $\Phi$. Then, $\Ec = \P^{-1}\Ac, \Fc = \P^{-1} \Ac^*$ are continuous, complementary subbundles of $\Vc$ of finite dimension and codimension, respectively. 
\end{prop}

We assume without further mention all the hypotheses of Proposition \ref{prop:compSubspace} for the remainder of \S \ref{sec:hard}. The following is an outline of the proof.

\begin{enumerate}
\item In \S \ref{subsubsec:attractor}, we show that when $\Ac$ is an asymptotically compact attractor for $\Phi$, we have that $\Ec := \P^{-1} \Ac$ is a continuous finite-dimensional subbundle of $\Vc$ (Lemma \ref{lem:attractorCtsSubbundle}).
\item In \S \ref{subsubsec:dualRepeller}, we show that the dual repeller $\Ac^*$ is of the form $\Fc = \P^{-1} \Ac^*$, where $\Fc \subset \Vc$ is a closed subset which meets each fiber $\Vc_b$ in a subspace complementary to $\Ec_b = \Vc_b \cap \Ec$. At this point, we have not yet shown that $\Fc$ is a continuous subbundle.
\item In \S \ref{subsubsec:exponentialSep}, we deduce that $\Ec, \Fc$ are exponentially separated with uniform estimates across all of $\Vc$.
\item In \S \ref{subsubsec:repellerContinuity}, we deduce the continuity of $\Fc$ from the exponential separation of $\Ec, \Fc$.
\end{enumerate}

\subsection{Attractors for linear semiflows}\label{subsubsec:attractor}

We first study attractors for the projectivized semiflow on $\P \Vc$. The proofs in this section follow of \cite{salamon1988flows} and Chapter 5 of \cite{colonius2012dynamics}.

Let $(\Ac, U)$ be an asymptotically compact attractor pair. Note that any $v \in \Vc$ for which $\P v \in \Ac$ has a negative continuation by Lemma \ref{lem:asympCompact}.

\begin{lem}\label{lem:attractorSubspace}
Let $(\Ac, U)$ be an asymptotically compact attractor pair for $\P \Phi$. Write $\Ac_b = \Ac \cap \P \Vc_b$ for $b \in B$.
\begin{itemize}
\item[(a)] For each $b \in B$, we have that $\P^{-1} \Ac_b \subset \Vc_b$ is a finite-dimensional linear subspace.
\item[(b)] For any $v, v' \in \Vc_b \setminus B \times \{0\}$, $\P v \in \Ac, \P v' \notin \Ac$ for which $v'$ has a negative continuation, we have that
\begin{align}\label{eq:weakDecayAttractor}
\lim_{t \to - \infty} \frac{|\Phi^t_b v|}{|\Phi^t_b v'|} = 0 \, .
\end{align}
\end{itemize}
\end{lem}

\begin{proof}
Without loss, let $v \in \P^{-1} \Ac_b, v' \in \Vc_b \setminus \P^{-1} \Ac$ be unit vectors, and assume that $\P v'$ has negative continuation. Throughout we let $L \subset \Vc_b$ denote the two-dimensional subspace of vectors spanned by $v, v'$. It follows by linearity that any vector in $L$ possesses a negative continuation. 

Let us assume in addition that $\P v$ is a boundary point of $\P L \cap \Ac$ relative to $\P L$: our first step is to prove \eqref{eq:weakDecayAttractor} in this special case. For this, we take on the following contradiction hypothesis:  
\[
\limsup_{t \to - \infty} \frac{|\Phi^t_b v|}{|\Phi^t_b v'|}  > 0 \, .
\]
Equivalently, there is a constant $K > 0$ and a subsequence $t_n \to \infty$ such that for any $n$,
\[
|\Phi^{-t_n}_b v'|  \leq K |\Phi^{-t_n}_b v| \, .
\]
%Observe that for $c \in \R$, $|\Phi^{- t_n}_b (v + c v')| \geq |\Phi^{- t_n}_b v | - |c| \cdot |\Phi^{- t_n}_b v'| \geq (1 - K |c|) |\Phi^{- t_n}_b v|$. So,
Let $c \in \R$ be arbitrary. We estimate:
\begin{align}\label{eq:projectiveNormComp}\begin{split}
d_\P&( \P \Phi^{- t_n}_b (v + c v'), \P \Phi^{- t_n}_b v) \leq \bigg| \frac{\Phi^{- t_n}_b (v + c v')}{|\Phi^{- t_n}_b (v + c v')|} - \frac{\Phi^{- t_n}_b v}{|\Phi^{- t_n}_b v|}  \bigg| \\
& \leq \frac{1}{|\Phi^{- t_n}_b (v + c v')| \cdot |\Phi^{- t_n}_b v|} \bigg( \bigg| |\Phi^{- t_n}_b v| \cdot \Phi^{- t_n}_b (v + c v') - |\Phi^{- t_n}_b (v + c v')| \cdot \Phi^{- t_n}_b (v + c v') \bigg| \\
& + \bigg|  |\Phi^{- t_n}_b (v + c v')| \cdot \Phi^{- t_n}_b (v + c v') - |\Phi^{- t_n}_b (v + c v')| \cdot \Phi^{- t_n}_b v \bigg|  \bigg)  \leq 2 |c| \frac{|\Phi^{- t_n} v'|}{|\Phi^{- t_n} v|} \, .
\end{split}\end{align}
Applying the contradiction hypothesis, we obtain
%The first term inside the parenthesis is
%\[
%\leq |\Phi^{- t_n}_b (v + c v')| \cdot |c| \cdot |\Phi^{- t_n}_b v'| \leq K |c| \cdot |\Phi^{- t_n}_b (v + c v')| \cdot |\Phi^{- t_n}_b v| \, .
%\]
%The second term is also $\leq  K |c| \cdot |\Phi^{- t_n}_b (v + c v')| \cdot |\Phi^{- t_n}_b v|$, and so we conclude that 
\[
d_\P( \P \Phi^{- t_n}_b (v + c v'), \P \Phi^{- t_n}_b v) \leq 2 K |c|
\]
for all $n$. Noting that $\P \Phi^{- t_n}_b v \in \Ac$ for all $n$, it follows that $\P \Phi^{- t_n}_b (v + c v') \in U$ for all $n$ when $|c|$ is chosen sufficiently small. Fixing such a $c$ and letting $v_n = \Phi_b^{- t_n} (v + c v')$, note that $\{\P v_n\} \subset U$, hence by asymptotic compactness it follows that all limit points of $\{\P \Phi_{\phi^{-t_n} b}^{t_n} \P v_n\}_n$ (of which there is at least one) belong to $\Ac$. But $\P \Phi^{t_n}_{\phi^{-t_n} b} \P v_n \equiv \P (v + c v')$ for all $n$, and so we deduce that $\P (v + c v') \in \Ac$ for all $c$ sufficiently small. This contradicts the assumption that $\P v$ is a boundary point of $\Ac \cap \P L$ in $\P L$. Thus \eqref{eq:weakDecayAttractor} holds for all such $v'$ in the case when $\P v$ is a boundary point of $\P L \cap \Ac$ relative to $\P L$.

\medskip

In the next step, we show that for any two-dimensional subspace $L \subset \Vc_b$, we have that $\Ac \cap \P L$ consists of a single point, if it contains a boundary point $\P v$ as above. Note first that either $\P v$ is the only point in $\P L$ with negative continuation, or that every point of $\P L$ has a negative continuation. In the former case there is nothing to prove, as every point of $\Ac$ possesses a negative continuation by Lemma \ref{lem:asympCompact}, part (c). Assuming the latter, let $\P v' \in \P L \setminus \Ac$ and note that any element of $\P L \setminus \Ac$ is of the form $v' + c v$ for some $c \in \R$. It follows from \eqref{eq:weakDecayAttractor} and a computation similar to that in \eqref{eq:projectiveNormComp} that
\begin{align}\label{eq:backwardsConv}
\lim_{t \to \infty} d_\P(\P \Phi^{- t}_b v', \P \Phi^{- t}_b(v' + c v) ) = 0 
\end{align}
for any $c \in \R$. 

Assume for the sake of contradiction that $\P (v' + c v) \in \Ac$ for some $c \in \R$. Then, \eqref{eq:backwardsConv} implies that $\P \Phi^{- t}_b v' \in U$ for all $t$ sufficiently large, hence (using asymptotic compactness and arguing as above) $\P v' \in \Ac$. This is a contradiction, so that $\P (v' + c v) \notin \Ac$ for any $c \in \R$. We conclude that $\Ac \cap \P L = \{\P v\}$, as desired.

\medskip

To complete the proof of part (a), note that we have shown that $\P L \cap \Ac$ is either empty, consists of a single point, or is nonempty and has an empty boundary in $\P L$. In this last case, we obtain automatically that $\P L \cap \Ac = \P L$ by the connectedness of $\P L$. We conclude that $\P^{-1} \Ac \cap \Vc_b$ is a linear subspace for all $b \in B$, and since $\Ac$ is compact, $\P^{-1} \Ac \cap \Vc_b$ must be finite dimensional as well. 

Finally, to check item (b), form the plane $L$ spanned by $v, v'$ and note that $\P L \cap \Ac = \{\P v\}$ by part (a), hence $\P v$ is a boundary point of $\Ac \cap \P L$ and so \eqref{eq:weakDecayAttractor} follows from the first part of the above proof.
\end{proof}

%We now show that when $B$ is chain transitive, $\P^{-1} A$ is a continuous subbundle of $\Vc$.

\begin{lem}\label{lem:attractorCtsSubbundle}
Assume that $B$ is chain transitive. Then, $\Ec := \P^{-1} \Ac$ is a continuous subbundle of $\Vc$ of constant finite dimension.
\end{lem}
\begin{proof}
We first show that if $B$ is chain transitive, then $\Ec_b = \Ec \cap \Vc_b$ has constant dimension independent of $b \in B$. It then follows from Lemma \ref{lem:compactSubbundle} that $\Ec$ is a continuous subbundle.

\medskip

We will show that for any $b, b' \in B$, we have $\dim \Ec_b \leq \dim \Ec_{b'}$. To start, observe that $\Omega(\Ac_b) \subset \Ac$ (Lemma \ref{lem:chain1}), and so $\Omega(\Ac_b) \cap \P \Vc_{b'} \subset \Ac_{b'}$; thus it suffices to prove that $\P^{-1} \Omega(\Ac_b) \cap \Vc_{b'}$ contains a subspace of dimension $\dim \Ec_b$.

For this, let $\e > 0$, and assume $T > 0$ is sufficiently large so that $\overline{\Phi([T, \infty) \times U)} \subset B_\e(\Ac)$ as in Lemma \ref{lem:asympCompact2}. Let $b_1, \cdots, b_n$, be an $(\e, T)$-chain from $b = b_0$ to $b' = b_{n + 1}$ with times $T_0, \cdots, T_n \geq T$, i.e., $d_B(\phi^{T_i} b_i, b_{i + 1}) < \e$ for all $0 \leq i \leq n$. 

Let now $v^1, \cdots, v^d \subset \Ec_b$ be a basis of $\Ec_b$, $d := \dim \Ec_b$. For each $1 \leq j \leq d$, the chain $b, b_1, \cdots, b_n, b'$ lifts to an $(\e, T)$-chain $(b_1, \P v_1^j), \cdots, (b_n, \P v_n^j)$ taking $(b, \P v^j)$ to $(b', \P \hat v^j)$ by setting $v_{i + 1}^j = \Phi^{T_i} (b_i, v_i^j)$, $v^j_0 := v^j$ and $\hat v^j := v^j_{n + 1}$ for $0 \leq i \leq n, 1 \leq j \leq d$.

By our choice of $\e, T$, it follows that $d_{\P}(\P \hat v^j, \Ac_{b'}) < 2 \e$. Moreover, by the injectivity of $\Phi$ it follows that $\{\hat v^j\}$ is linearly independent. 

Collecting, we have shown that for any $\e > 0$ and $T = T(\e)$ sufficiently large, $\P^{-1} \big( \Omega(\Ac_b; \e, T) \big) \cap \Vc_{b'}$ contains a $d$-dimensional subspace $E_\e$, and that by construction, $\P E_\e \subset B_{2 \e}(\Ac_{b'})$. 

To complete the proof, fix a sequence $T_n \to \infty$ for which $T_n \geq T(1/n)$. For each $n$ let $E_{1/n} \subset B_{2 / n}(\Ac_{b'})$ denote the $d$-dimensional subspace constructed above, and let $\{w_n^1, \cdots, w_n^d\} \subset E_{1/n}$ be a basis of unit vectors for which $N[w_n^1, \cdots, w_n^d] \leq C_d$, where $C_d$ depends only on $d \in \N$ (Lemma \ref{lem:compExist}). For each $n$ and $1 \leq i \leq d$ there is a unit vector $\hat w^i_n \in \P^{-1} \Ac_{b'}$ for which $|\hat w_n^i - w_n^i| \leq 2/n$; thus, when $n$ is sufficiently large, it holds that $\{\hat w_n^1, \cdots, \hat w_n^d\} \subset \P^{-1}\Ac_{b'}$ are linearly independent-- this follows from the estimates in the proof of Lemma \ref{lem:projectionContinuity} (a) and the uniform estimate on $N[w_n^1, \cdots, w_n^d]$. Thus we have obtained $\dim \P^{-1} \Ac_{b'} \geq d$, as desired.
\end{proof}

\subsection{Dual repeller subspaces}\label{subsubsec:dualRepeller}

We now turn to the repeller $\Ac^*$ for $\Ac$. 

\medskip

\begin{lem}\label{lem:repellerSubspace}
Let $\Ac$ be the attractor of an asymptotically compact preattractor $U$, and let $\Ac^*$ be its dual repeller. Write $\Ac^*_b = \Ac^* \cap \P \Vc_b$.
\begin{itemize}
\item[(a)] For any $b \in B$, $\P^{-1}\Ac^*_b$ is a linear subspace of $\Vc_b$.
\item[(b)] For any $\P v \in \Ac^*_b,\P v' \notin \Ac^*_b$, we have
\begin{align}\label{eq:weakDecayRepeller}
\lim_{t \to \infty} \frac{|\Phi^t_b v|}{|\Phi^t_b v'|} = 0 \, .
\end{align}
\end{itemize}
\end{lem}
\begin{proof}
This proof follows that of Lemma \ref{lem:attractorSubspace}; indeed, it is somewhat simpler, since we need not concern ourselves with the existence of negative continuations. 

To begin, let $\P v \in \Ac^*_b, \P v' \in \Vc_b \setminus \Ac^*_b$, and form the two-dimensional subspace $L \subset \Vc_b$ spanned by $v, v'$. Assuming $\P v$ is a boundary point of $\Ac^* \cap \P L$ relative to $\P L$, we will show that \eqref{eq:weakDecayRepeller} holds. 

If it does not, then as before there is a sequence of positive reals $t_n \to \infty$ and a constant $K > 0$ such that
\begin{align}\label{eq:contra1}
|\Phi^{t_n}_b v'| \leq K |\Phi^{t_n}_b v| 
\end{align}
for all $n$. Following the time-reversed analogue of the computation in \eqref{eq:projectiveNormComp}, we conclude that
\begin{align}\label{eq:vpcvCloseToAttractor}
d_\P( \P \Phi^{t_n}_b (v + c v'), \P \Phi^{t_n}_b v) \leq 2 K |c|
\end{align}
for arbitrary $c \in \R$. From here on, fix $\e > 0$ so that $B_\e(\Ac) \subset U$; we assume in what follows that $|c| \ll \e / 2 K$, so that $d_\P( \P \Phi^{t_n}_b (v + c v'), \P \Phi^{t_n}_b v) < \e/2$ for all $n$.

Recalling that $\P v \in \Ac^* \cap \P L$ is a boundary point, there is some $c \in [-\e / 2K, \e / 2 K] \setminus \{0\}$ such that $v + c v' \notin \Ac^*$. Fixing such a $c$, by definition $\omega(v + c v') \cap \Ac \neq \emptyset$ and so there is a sequence $t_n' \to \infty$ for which $\{\P \Phi^{t_n'}_b(v + c v')\}$ converges to a point of $\Ac$; by the definition of preattractor we conclude that there exists $T > 0$ such that for any $t \geq T$, $\P \Phi^t_b(v + c v') \in U$. By asymptotic compactness it follows that a subsequence $\{\P \Phi^{t_{n_j}}_b (v + c v')\}$ converges to a point in $\Ac$. 

In particular, $\P \Phi^{t_{n_j}}_b(v + c v') \in B_{\e/2}(\Ac)$, hence by \eqref{eq:vpcvCloseToAttractor} we have $\P \Phi^{t_{n_j}}_b(v) \in B_\e(\Ac) \subset U$ for $j$ sufficiently large. But now, $\{\P \Phi^{t_{n_j}}_b(v)\}$ possesses a subsequence converging to a point of $\Ac$ by asymptotic compactness, which contradicts the assumption that $v \in \Ac^*$. Thus \eqref{eq:weakDecayRepeller} holds in the case when $\P v$ is a boundary point of $\P L \cap \Ac^*$.

\medskip

Next, we show that if $\Ac^* \cap \P L$ contains a boundary point $\P v$ as above, then $\Ac^* \cap \P L = \{\P v\}$ consists of a single point. For this, fix such a boundary point $\P v$ and let $\P v' \in \Vc_b \setminus \Ac^*$. Applying \eqref{eq:weakDecayRepeller} to this choice of $\P v, \P v'$, we deduce that
\[
\lim_{t \to \infty} d_\P(\P \Phi^{t}_b v', \P \Phi^{t}_b(v' + c v) ) = 0 
\]
for all $c \in \R$, following the computation \eqref{eq:projectiveNormComp} in Lemma \ref{lem:attractorSubspace}. Since $\omega(\P v') = \omega(\P (v' + c v))$, we conclude that $v' + c v \notin \Ac^*$ for any $c \in \R$; in particular $\Ac^* \cap \P L = \{\P v\}$, as desired.

\medskip

To complete the proof of (a), note that for any two-dimensional subspace $L \subset \Vc_b$ that $\P L \cap \Ac^*$ is either empty, a single point, or all of $\P L$-- this implies that $\Ac^* \cap \Vc_b$ is a subspace for any $b \in B$, which is a closed subspace by the fact that $\Ac^* \subset \P \Vc$ is closed. Part (b) follows for any $\P v, \P v' \in \Vc_b$ with $\P v \in \Ac^*, \P v' \notin \Ac^*$ by considering the two-dimensional subspace $L \subset \Vc_b$ spanned by $v, v'$.
\end{proof}

We now deduce that the dual repeller to $\Ac$ is a complementary subbundle of codimension equal to the dimension of $\Ac$. { Here we significantly deviate from the finite-dimensional proof, as we must carefully argue around the fact that $\P \Vc$ is not locally compact.}

\begin{lem}\label{lem:repellerComplements}
We have that $\Fc = \P^{-1} \Ac^*$, where for each $b \in B$ we have that $\Fc_b = \Fc \cap \Vc_b$ is a complement to $\Ec_b$ for which $|\pi_{\Ec_b \ds \Fc_b}| \leq C$, where $C > 0$ is independent of $b \in B$.
\end{lem}
\begin{proof}
Fix $b \in B$: we will show that $\P^{-1} \Ac^*_b$ is a closed, finite codimensional complement to $\Ec_b$. To start, using Lemma \ref{lem:compExist} fix for each $n$ a complement $F_n'$ to $\Ec_{\phi^n b}$ for which $|\pi_{\Ec_{\phi^n b}\ds F_n'}| \leq \sqrt{\dim \Ec} + 2$. Then, by Lemma \ref{lem:projectiveEstimate}, there is some $\e > 0$, depending only on $\dim \Ec$, for which $F_n' \subset V_\e$ for all $n$. Fix such an $\e$.

%\begin{cla}\label{cla:existComplement}
%For any $\e > 0$ sufficiently small (in terms of the constant finite dimension of $\Ec$) and for any $n$, there exists a complement $\hat \Fc \subset \Vc$ of finite codimension such that for each $b \in B$, $\hat \Fc_b = \Vc_b \cap \hat \Fc$ complements $\Ec_b$, and $\P \hat \Fc \subset V_\e$.
%\end{cla}
%\begin{proof}
%The proof of Claim \ref{cla:existComplement} follows from continuity of $\Ec$ and an argument `stitching together' pieces of the subbundle into a continuous subbundle.
%\end{proof}

One now checks that for all $n \geq 1, b \in B$, the preimage $F_n  := (\Phi^{n}_{ b})^{-1} F_n'$ is a subspace complementary to $\Ec_b$. This is straightforward: the bounded projection operator $\pi_n := (\Phi^n_{ b}|_{\Ec_{ b}})^{-1} \circ \pi_{\Ec_{\phi^n b} \ds F_n'} \circ \Phi^n_{ b}$ has image $\Phi^n_{ b} \Ec_{ b} = \Ec_{ \phi^n b}$ and kernel $F_n = (\Phi^n_{ b})^{-1} F_n'  = \{f \in \Vc_b : \Phi^n_b f \in F_n'\}$ (for more details, see Lemma 2.4 in \cite{Blumenthal20162377}).

Since $\P F_n' \subset V_\e$ for all $n$, it follows from Lemma \ref{lem:repellersRepel} that $\P F_n \subset V_\e$ for all $n \geq T = T(\e)$. In particular, for all $b \in B$ and $n \geq T(\e)$ we have that $|\pi_n|$ is bounded from above by a constant $C = C(\e) > 0$ by Lemma \ref{lem:projectiveEstimate} and \eqref{eq:minAngle}.

Fixing a complement $F$ to $\Ec_b$ in $\Vc_b$, define
\[
G_n = \pi_n|_{F} \, ,
\]
so that $\graph G_n = \{f + G_n(f) : f \in F\} = F_n$ for all $n$.

Observe that $|G_n| \leq C$ for all $n$. We now appeal to the following Lemma.

\begin{lem}\label{lem:arzela}
Let $V$ be a separable Banach space. Let $d \in \N$, and let $\{G_n\} \subset L(V, \R^d)$ be an infinite collection of bounded linear maps for which $|G_n| \leq C$ for all $n$, where $C > 0$ is a constant. Then, there is a subsequence $\{n_i\}$ along which $\{G_{n_i}\}$ converges in the strong operator topology on $L(V, \R^d)$ to some $G \in L(V, \R^d)$-- that is, for any fixed $v \in V$, we have that $G_{n_i} v \to G v$.
\end{lem}
\begin{proof}
By the Banach-Alaoglu Theorem, the unit ball of $\Bc^*$ is compact in the weak$^*$ topology. Since $\Bc^*$ is metrizable when $\Bc$ is separable, it follows that for any sequence of unit vectors $\{ l_n \} \subset \Bc^*$ there is a weak$^*$ convergent subsequence $\{ l_{n'}\}$. One then applies this argument to each of the $d$ coordinate functionals comprising $G_n : V \to \R^d$, obtaining a subsequence $G_{n_i}$ which converges in the strong operator topology.
\end{proof}

Regarding $\{G_n\}$ as a sequence of linear operators $F \to  \Ec_b \cong \R^{\dim \Ec}$, we have satisfied the setup of Lemma \ref{lem:arzela}. Thus there is a sequence $n_i \to \infty$ and a bounded linear operator $G : F \to \Ec_b$ such that  $G(f) = \lim_i G_{n_i}(f)$ for all $f \in F$. 

We claim that $\graph G = \P^{-1} \Ac^*_b$. To show `$\subset$', fix $f \in \hat F \setminus \{0\}$ and write $v_n = f + G_n(f)$, so that $v_n \to v \in \graph G$ where $v = f + G f$. Since $v_n \in F_n$, by construction $\Phi^n_b v_n \in F_n'$ for all $n$, and so $\P \Phi^n_b v_n \in V_\e$. Thus $v_n \in (\Phi^n_b)^{-1} V_\e$, and so
\[
\P v \in \bigcap_{t \geq 0} \overline{\bigcup_{s \geq t} (\P \Phi^s)^{-1} V_\e} \, ,
\]
hence $\P v \in  \Ac^*$ by Lemma \ref{lem:repellersRepel}.

For the opposite inclusion, let $v \in \Vc_b \setminus \graph G$ and observe that $\graph G$ complements $\Ec_b$ in $\Vc_b$, hence $v = e + f$ for some $e \in \Ec_b, f \in \graph G \subset \Ac^*_b$. Since $v \notin \graph G$, we have $e \neq 0$. Thus $d_{\P} (\P \Phi^t_b v, \P \Phi^t_b e) \to 0$ as $t \to \infty$ by Lemma \ref{lem:repellerSubspace}, which implies that $\omega(\P v) \cap \Ac \neq \emptyset$ by asymptotic compactness. Thus $\P v \notin \Ac^*_b$. As $v \in \Vc_b \setminus \graph G$ was arbitrary, we conclude that $\P^{-1} \Ac^*_b \subset \graph G$.

%\medskip
%
%{ It now remains to show that $\Fc = \P^{-1}A^*$ is a continuous subbundle; this follows from Lemma \ref{lem:closedSubbundle2}.}
\end{proof}

\subsection{Deducing exponential separation}\label{subsubsec:exponentialSep}

We now show that $\Ec_b, \Fc_b$ are exponentially separated with uniform constants.

To begin, we show the following.

\begin{lem}\label{lem:oneHalf}
There exists $T > 0$ such that for any $b \in B$ and any unit vectors $e \in \Ec_b$, $f \in \Fc_b$, we have that
\[
|\Phi^T_b f| \leq \frac12 |\Phi^T_b e| \, .
\]
\end{lem}
\begin{proof}
Let $e \in \Ec_b, f \in \Fc_b$ be any two unit vectors. Using compactness of $\Ac$, let $\e > 0$ be such that $B_\e(\Ac) \subset U$. Assume without loss that $\e \leq 1/3$. In particular, note by Lemma \ref{lem:asympCompact2} that there exists $T$ such that $\Phi([T, \infty) \times B_\e(A)) \subset B_{\e / C^*}(\Ac)$ for this choice of $\e$. Here we take $C^* = \sup_{b \in B} |\pi_{\Fc_b \ds \Ec_b}|$, which by Lemma \ref{lem:repellerComplements} is finite. This will be the value of $T$ as in the statement of Lemma \ref{lem:oneHalf}.

Form $v = e + \a f$, where $\a > 0$ is chosen so that $\P v \in B_\e(\Ec_b)$. For this it suffices, by Lemma \ref{lem:projectiveEstimate}, to take $\a$ so that
\[
\frac{2 \a}{1 - \a} \leq \e \, .
\]
Now, set $e_T = \Phi^T_b e, f_T = \Phi^T_b f$. By construction, $v_T = \Phi^T_b v$ is such that $\P v_T \in B_{\e / C^*}(\Ec_{\phi^T b})$, hence
\[
\frac{|f_T|}{|v_T|} = \frac{|\pi_{\Fc_{\phi^T b} \ds \Ec_{\phi^T b}} v_T|}{|v_T|} \leq |\pi_{\Fc_{\phi^T b} \ds \Ec_{\phi^T b}}| \cdot d_\P(\P v_T, \P \Ec_{\phi^T b}) \leq \e
\]
by Lemma \ref{lem:projectiveEstimate}. Rearranging and applying the triangle inequality (i.e., $|v_T| \leq |e_T| + |f_T|$), we obtain
\[
|f_T| \leq \frac{\e}{1 - \e} |e_T|  \leq \frac12 |e_T|
\]
by our stipulation that $\e \leq 1/3$.
\end{proof}

\begin{lem}
There are constants $K > 0, \gamma > 0$ such that for any $b \in B, t \geq 0$, 
\begin{align}\label{eq:separateNorms}
|\Phi^t_b|_{\Fc_b}| \leq K e^{- \gamma t} m(\Phi^t_b|_{\Ec_b})  \, .
\end{align}
\end{lem}

\begin{proof}
From Lemma \ref{lem:oneHalf}, observe that
\[
\frac{|\Phi^{k T}_b f|}{|\Phi^{k T}_b e|} \leq \frac12 \frac{|\Phi^{(k-1) T}_b f|}{|\Phi^{(k-1) T}_b e|} \leq \cdots \leq \bigg( \frac{1}{2} \bigg)^k 
\]
for any unit vectors $e \in \Ec_b, f \in \Fc_ b$, $k \in \N$, where $T$ is as in Lemma \ref{lem:oneHalf}. Thus
\[
|\Phi^{k T}_b|_{\Fc_b}| \leq 2^{-k} m(\Phi^{k T}_b |_{\Ec_b})
\]
for all $k \in \N$.

By an argument using the Steinhaus Uniform Boundedness Principle, it follows that
\[
C_1 = \sup_{\substack{b \in B \\ 0 \leq t \leq T}} |\Phi^t_b| < \infty \, .
\]
By the continuity of $b \mapsto \Ec_b$ and finite dimenisonality, we have as well that
\[
\inf_{\substack{b \in B \\ 0 \leq t \leq T}} m(\Phi^t_b|_{\Ec_b}) =: C_2 > 0 \, .
\]
Now, if $t = k T + s$ for some $0 \leq s < T$, we estimate $|\Phi^t_b|_{\Fc_b}| \leq C_1 2^{-k} m(\Phi^{kT}_b|_{\Ec_b})$. Noting that $m(\Phi^t_b|_{\Ec_b}) \geq m(\Phi^s_{\phi^{k T} b}|_{\Ec_{\phi^{k T} b}}) \cdot m(\Phi^{k T}_b|_{\Ec_b}) \geq C_2 m(\Phi^{k T}_b|_{\Ec_b})$, it follows that
\[
|\Phi^t_b |_{\Fc_b}| \leq  C_1 C_2^{-1} 2^{-k} m(\Phi^t_b|_{\Ec_b}) \, .
\]
Thus, \eqref{eq:separateNorms} holds with 
\[
\gamma = \frac{\log 2}{ 2 T}  \quad \text{ and } K = \frac{C_1}{ C_2} \, .
\]

\end{proof}

\subsection{Continuity of the repeller subspaces}\label{subsubsec:repellerContinuity}

At last, we deduce the continuity of $b \mapsto \Fc_b$ in the Hausdorff distance $d_H$.

\begin{lem}\label{lem:fcContinuous}
The assignment $b \mapsto \Fc_b$ is continuous in the Hausdorff distance $d_H$.
\end{lem}
\begin{proof}
Write $\pi_b = \pi_{\Ec_b \ds \Fc_b}$ for $b \in B$. For $b, b' \in B$ sufficiently close, we will obtain a bound on $|\pi_{\Ec_b \ds \Fc_b}|_{\Fc_{b'}}|$.

Let $v \in \Fc_{b'}$ be a unit vector. Then
\begin{align*}
|\pi_b v| \cdot m(\Phi^n_b|_{\Ec_b}) & \leq |\Phi^n_b \circ \pi_b v| = |\pi_{\phi^n b} \circ \Phi^n_b v| \leq \big( \sup_{b \in B} |\pi_b| \big) \cdot |\Phi^n_b v| \\
& \leq C' \cdot \big( |\Phi^n_b - \Phi^n_{b'}| + |\Phi^n_{b'} v| \big) \, .
\end{align*}
Here, $C' = \sup_{b \in B} |\pi_x| < \infty$ by Lemma \ref{lem:repellerComplements}. Given $\e > 0$, fix $n$ for which $2 C' K e^{- n \gamma} < \e $; with this value of $n$ fixed, let $\d > 0$ be such that if $d_B(b, b') < \d$, then $|\Phi_b^n - \Phi_{b'}^n| < |\Phi^n_b|_{\Fc_b}|$ (the value of which may, a priori, depend on $b$). Plugging all this in,

\[
|\pi_b v| \leq 2 C'  \frac{|\Phi^n_b|_{\Fc_b}|}{m(\Phi^n_b|_{\Ec_b})} \leq 2 C' K e^{- n \gamma} < \e \, .
\]
Since $v$ was arbitrary, we conclude that
\[
\Gap(\Fc_{b'}, \Fc_b) \leq |\pi_b|_{\Fc_{b'}}| < \e
\]
whenever $d_B(b, b') < \d$.

Assuming, as we may, that $\e \ll 1/d$, where $d = \dim \Ec$, it follows from Lemma \ref{lem:symmCloseness} that $\Gap(\Fc_b, \Fc_{b'}) \leq d \e / (1 - d \e) \leq 2 d \e$. By \eqref{eq:gapEquiv}, we conclude that $d_H(\Fc_b, \Fc_{b'}) \leq 4 d \e$. This completes the proof.
\end{proof}

\section{Completing the proofs of Theorems \ref{thm:equiv} and \ref{thm:finestDecomp}}\label{sec:cleanUp}

Throughout we are in the setting of Theorems \ref{thm:equiv}, \ref{thm:finestDecomp}.

\subsection{Completing the proof of Theorem \ref{thm:equiv}}

In \S \ref{sec:hard}, we showed that an attractor-repeller pair $\Ac, \Ac^*$ gives rise to an exponential splitting $\Vc = \Ec \oplus \Fc$, where $\P \Ec = \Ac, \P \Fc = \Ac^*$. Below we prove the converse implication. 

\begin{prop}\label{prop:easyConverse}
Let $\Vc = \Ec \oplus \Fc$ be an exponential splitting. Then $\Ac = \P \Ec$ is an asymptotically compact attractor for the projectivized flow $\P \Phi$.
\end{prop}
\begin{proof}
Let $\Ac = \P \Ec$. By Lemma \ref{lem:asympCompact2}, it suffices to show that $B_\e(\Ac)$ is a preattractor for all $\e > 0$ sufficiently small. This we obtain by showing the following: for any $\e > 0$ sufficiently small, there exists $T = T_\e > 0$ such that for any $b \in B$, $\P v \in B_\e(\Ac)$, we have that
\[
d_\P(\P \Phi^t_b v, \Ac_{\phi^t b}) \leq \e /2 
\]
for all $t \geq T_\e$. Here $v \in \Vc_b$ is a unit vector representative for $\P v \in \P \Bc$.

Let $\e > 0$, which we will adjust smaller a finite number of times in the following proof. Let us write $v = e + f$ and $v_t = \Phi^t_b v = e_t + f_t$ according to the splittings $\Ec_b \oplus \Fc_b$ and $\Ec_{\phi^t b} \oplus \Fc_{\phi^t b}$, respectively.  Using Lemma \ref{lem:projectiveEstimate}, we estimate
\[
(*) = d_\P(\P \Phi^t_b v, \Ac_{\phi^t b}) \leq 2 \frac{|\pi_{\Fc_b \ds \Ec_b} \Phi^t_b v|}{|\Phi^t_b v|} = 2 \frac{|f_t|}{|v_t|} \leq 2 \frac{|f_t|}{|e_t| - |f_t|} = 2 h\bigg(\frac{|f_t|}{|e_t|}\bigg) \, ,
\]
where $h(r) = \frac{r}{1 - r}$ is an increasing function $[0,1) \to [0,\infty)$. Now, exponential separation implies that
\[
\frac{|f_t|}{|e_t|} \leq K e^{- \gamma t} \frac{|f_0|}{|e_0|} \, .
\]
Finally, we observe that $|e_0| \geq 1 - |f_0|$, hence $\frac{|f_0|}{|e_0|} \leq h(|f_0|)$, and that $|f_0| \leq \e |\pi_{\Fc_b \ds \Ec_b}|$ by Lemma \ref{lem:projectiveEstimate}. Collecting, we have that
\[
(*) \leq 2 h \big( K e^{- \gamma t} h(\e |\pi_{\Fc_b \ds \Ec_b}|) \big) \, .
\]

Taking $\e \leq \min\{ 1, 1/(10 C')\}$, where $C' := \sup_{b \in B} |\pi_{\Fc_b \ds \Ec_b}| < \infty$, yields $(*) \leq 2 h(2 K C' e^{- \gamma t} \e)$. Letting $T = T_\e > 0$ be sufficiently large so that $2 K e^{- \gamma T}  \leq 1/10$, 
\[
(*) \leq 8 K C' e^{- \gamma t} \e \, ,
\]
which is $\leq \e / 2$ when $T$ is chosen still larger so that $8 K C' e^{- \gamma T} \leq 1/2$.

\end{proof}

\subsection{Proof of Theorem \ref{thm:finestDecomp}}

The plan for the proof of Theorem \ref{thm:finestDecomp} is as follows.
\begin{enumerate}
\item In \S \ref{subsubsec:algorithm}, we present an algorithm for constructing the attractor sequence $\{ \Ac_i\}$ as in the statement of Theorem \ref{thm:finestDecomp}.
\item In \S \ref{subsubsec:checkAlgorithm}, we check that the algorithm from \S \ref{subsubsec:algorithm} produces an attractor sequence with the property (b) in Theorem \ref{thm:finestDecomp}, namely, that $\{\Ac_i\}$ is the `finest' attractor sequence.
\end{enumerate}

%We begin by presenting an algorithm for the attractor sequence $\Ac_1 \subset \Ac_2 \subset \cdots$.

\subsubsection{An algorithm for producing the `finest' attractor sequence $\Ac_1 \subset \Ac_2 \subset \cdots$}\label{subsubsec:algorithm}

We begin by defining
\[
k_1 = \inf \{k \in \N : \Phi \text{ has an exponential separation of index } k \} \, ,
\]
where by convention we set $k_1 = \infty$ if the $\inf$ is taken over an empty set (i.e. no exponential separation exists). If $k_1 = \infty$ then we set $N = 0$ and terminate the procedure; otherwise we let $\Vc = \Vc_1 \oplus \Vc_1^-$ be the (unique; see Lemma \ref{lem:expSepUnique}) exponential separation of index $k_1$ for $\Phi$. We now define $\Ac_1 := \P \Vc_1$, which by Theorem \ref{thm:equiv} is an asymptotically compact attractor.

We now proceed by setting
\[
k_2 = \inf \{k \in \N : \Phi|_{\Vc_1^-} \text{ has an exponential separation of index } k\} \, .
\]
If $k_2 = \infty$ then we set $N = 1$ and terminate the procedure; otherwise we let $\Vc_1^- = \Vc_2 \oplus \Vc_2^-$ denote the (unique) exponential separation for $\Phi|_{\Vc_1^-}$ of index $k_2$. We now define $\Ac_2 := \P (\Vc_1 \oplus \Vc_2)$. It is quite clear that $\Vc_2^+ := \Vc_1 \oplus \Vc_2$ is exponentially separated from $\Vc_2^-$, and so it follows from Theorem \ref{thm:equiv} that $\Ac_2$ is an asymptotically compact attractor.

We now describe the inductive step: assuming the procedure has not been terminated by step $n-1$, let $\{k_i\}_{i = 1}^{n-1} \subset \N$ and $\Vc_1, \Vc_2, \cdots, \Vc_{n-1}$ and $\Vc_{n-1}^-$ be as above. We set
\[
k_n = \inf\{k \in \N : \Phi|_{\Vc_{n-1}^-} \text{ has an exponential separation of index } k\} \, .
\]
If $k_n = \infty$ we set $N = n-1$ and terminate; otherwise we let $\Vc_n^- = \Vc_n \oplus \Vc_{n}^-$ denote the exponential separation for $\Phi_{\Vc_{n-1}^-}$ of index $k_n$. We set $\Ac_n = \P(\Vc_1 \oplus \cdots \oplus \Vc_n)$, which as before is an asymptotically compact attractor.

If at each stage $n$ we have $k_n < \infty$, then the algorithm proceeds indefinitely and we set $N = \infty$. This completes the description of the algorithm.

\subsubsection{Checking the algorithm works}\label{subsubsec:checkAlgorithm}

The following is a reformulation of part (b) of Theorem \ref{thm:finestDecomp}.
\begin{lem}\label{lem:finestDecomp}
Let $N \in \N \cup \{\infty\}, \{\Ac_i\}_{i = 0}^N$ be as in \S \ref{subsubsec:algorithm}. If $\Ac$ is any nonempty asymptotically compact attractor, then $\Ac = \Ac_i$ for some $1 \leq i < N + 1$. 
\end{lem}

\begin{proof}%[Proof of Lemma \ref{lem:finestDecomp} assuming Proposition \ref{prop:moveDownSubspace}]
Let us define 
\begin{align*}
\hat k_1 &= \inf\{k \in \N : \Phi \text{ has an exponential separation of index } k \} \, , \quad \text{ and inductively,} \\
\hat k_n & = \inf \{k > k_{n-1} : \Phi \text{ has an exponential separation of index } k\} \, .
\end{align*}
where as usual the $\inf$ of an empty set is $\infty$. In this construction we set $\hat N = n$ to be the first stage $n$ for which $\hat k_n = \infty$, and set $\hat N = \infty$ if this never occurs.

To prove Lemma \ref{lem:finestDecomp}, it suffices by Lemma \ref{lem:expSepUnique} to show that $\hat N = N$ and $\hat k_n = k_1 + \cdots + k_n$ for all $1 \leq n < N + 1$. If $N = 0$, then $\hat N = 0$ clearly holds and there is nothing to check. Otherwise, $\hat k_1 = k_1$ by definition and $N, \hat N \geq 1$. 

Continuing, note that if $N = 1$ then $k_2 = \infty$; by Proposition \ref{prop:moveDownSubspace} we conclude $\hat k_2 = \infty$ and thus $\hat N = \infty$. Otherwise, $N, \hat N \geq 2$ and $\hat k_2 = k_1 + k_2$ by Proposition \ref{prop:moveDownSubspace}.

The induction hypothesis is that $N, \hat N \geq n-1$ and $\hat k_{l} = k_1 + \cdots + k_l$ for all $l \leq n -1$. If $N = n$, then $\hat k_n = \infty$ and $\hat N = n$ as before. Otherwise $\hat N, N \geq n$ and $\hat k_n = k_1 + \cdots + k_n$. This completes the proof.
\end{proof}

\bibliography{selBiblio}

\begin{thebibliography}{10}

\bibitem{babin2006global}
Anatoli{\u\i}~V Babin.
\newblock Global attractors in {PDE}.
\newblock In {\em Handbook of dynamical systems, Vol. 1B}, pages 983--1085.
  Elsevier, 2006.

\bibitem{babin1992attractors}
Anatoli{\u\i}~V Babin and Mark~I Vishik.
\newblock {\em Attractors of evolution equations}, volume~25.
\newblock Elsevier, 1992.

\bibitem{Blumenthal20162377}
Alex Blumenthal.
\newblock A volume-based approach to the multiplicative ergodic theorem on
  {B}anach spaces.
\newblock {\em Discrete and Continuous Dynamical Systems}, 36(5):2377--2403,
  2016.

\bibitem{blumenthal2015characterization}
Alex Blumenthal and Ian~D Morris.
\newblock Characterization of dominated splittings for operator cocycles acting
  on {B}anach spaces.
\newblock {\em arXiv preprint arXiv:1512.07602}, 2015.

\bibitem{Blumenthal2016}
Alex Blumenthal and Lai-Sang Young.
\newblock Entropy, volume growth and {SRB} measures for {B}anach space
  mappings.
\newblock {\em Inventiones mathematicae}, pages 1--61, 2016.

\bibitem{bochi2009some}
Jairo Bochi and Nicolas Gourmelon.
\newblock Some characterizations of domination.
\newblock {\em Mathematische Zeitschrift}, 263(1):221--231, 2009.

\bibitem{bronsteinnonautonomous}
Idel~U Bronstein.
\newblock Nonautonomous dynamical systems.
\newblock {\em Shtiintsa, chisnav (in Russian)}, 1984.

\bibitem{carvalho2013attractors}
Alexandre Carvalho, Jos\'{e}~A Langa, and James Robinson.
\newblock {\em Attractors for infinite-dimensional non-autonomous dynamical
  systems}.
\newblock Number 182 in Applied Mathematical Sciences. Springer-Verlag New
  York, 2013.

\bibitem{chen2011state}
Xiaopeng Chen and Jinqiao Duan.
\newblock State space decomposition for non-autonomous dynamical systems.
\newblock {\em Proceedings of the Royal Society of Edinburgh: Section A
  Mathematics}, 141(05):957--974, 2011.

\bibitem{chicone1999evolution}
Carmen Chicone and Yuri Latushkin.
\newblock {\em Evolution semigroups in dynamical systems and differential
  equations}.
\newblock Number~70 in Mathematical Surveys and Monographs. American
  Mathematical Soc., 1999.

\bibitem{choi2002chain}
Sung~Kyu Choi, Chin-Ku Chu, and Jong~Suh Park.
\newblock Chain recurrent sets for flows on non-compact spaces.
\newblock {\em Journal of Dynamics and Differential Equations}, 14(3):597--611,
  2002.

\bibitem{chow1995existence}
Shui-Nee Chow and Hugo Leiva.
\newblock Existence and roughness of the exponential dichotomy for skew-product
  semiflow in {B}anach spaces.
\newblock {\em Journal of Differential Equations}, 120(2):429--477, 1995.

\bibitem{chow1996two}
Shui-Nee Chow and Hugo Leiva.
\newblock Two definitions of exponential dichotomy for skew-product semiflow in
  {B}anach spaces.
\newblock {\em Proceedings of the American Mathematical Society},
  124(4):1071--1081, 1996.

\bibitem{chow1996unbounded}
Shui-Nee Chow and Hugo Leiva.
\newblock Unbounded perturbation of the exponential dichotomy for evolution
  equations.
\newblock {\em Journal of Differential Equations}, 129(2):509--531, 1996.

\bibitem{colonius2012dynamics}
Fritz Colonius and Wolfgang Kliemann.
\newblock {\em The dynamics of control}.
\newblock Springer Science \& Business Media, 2012.

\bibitem{grune2000uniform}
Lars Gr{\"u}ne.
\newblock A uniform exponential spectrum for linear flows on vector bundles.
\newblock {\em Journal of Dynamics and Differential Equations}, 12(2):435--448,
  2000.

\bibitem{hale2010asymptotic}
Jack~K Hale.
\newblock {\em Asymptotic behavior of dissipative systems}.
\newblock Number~25 in Mathematical Surveys and Monographs. American
  Mathematical Soc., 2010.

\bibitem{henry1981geometric}
Daniel Henry.
\newblock {\em Geometric theory of semilinear parabolic equations}.
\newblock Number 840 in Lecture Notes in Mathematics. Springer, 1981.

\bibitem{hurley1991chain}
Mike Hurley.
\newblock Chain recurrence and attraction in non-compact spaces.
\newblock {\em Ergodic Theory Dynam. Systems}, 11(4):709--729, 1991.

\bibitem{hurley1992noncompact}
Mike Hurley.
\newblock Noncompact chain recurrence and attraction.
\newblock {\em Proceedings of the American Mathematical Society},
  115(4):1139--1148, 1992.

\bibitem{hurley1995chain}
Mike Hurley.
\newblock Chain recurrence, semiflows, and gradients.
\newblock {\em Journal of Dynamics and Differential Equations}, 7(3):437--456,
  1995.

\bibitem{johnson1980analyticity}
Russell~A Johnson.
\newblock Analyticity of spectral subbundles.
\newblock {\em Journal of differential equations}, 35(3):366--387, 1980.

\bibitem{johnson1987ergodic}
Russell~A Johnson, Kenneth~J Palmer, and George~R Sell.
\newblock Ergodic properties of linear dynamical systems.
\newblock {\em SIAM journal on mathematical analysis}, 18(1):1--33, 1987.

\bibitem{kato2013perturbation}
Tosio Kato.
\newblock {\em Perturbation theory for linear operators}, volume 132.
\newblock Springer Science \& Business Media, 2013.

\bibitem{kloeden2011nonautonomous}
Peter~E Kloeden and Martin Rasmussen.
\newblock {\em Nonautonomous dynamical systems}.
\newblock Number 176 in Mathematical Surveys and Monographs. American
  Mathematical Soc., 2011.

\bibitem{magalhaes1987spectrum}
Luis~T Magalhaes.
\newblock The spectrum of invariant sets for dissipative semiflows.
\newblock In {\em Dynamics of Infinite Dimensional Systems}, pages 161--168.
  Springer, 1987.

\bibitem{mane1983lyapounov}
Richardo Ma{\~n}{\'e}.
\newblock {L}yapounov exponents and stable manifolds for compact
  transformations.
\newblock In {\em Geometric dynamics}, pages 522--577. Springer, 1983.

\bibitem{oseledets1968multiplicative}
Valery~Iustinovich Oseledets.
\newblock A multiplicative ergodic theorem. characteristic {L}japunov exponents
  of dynamical systems.
\newblock {\em Trudy Moskovskogo Matematicheskogo Obshchestva}, 19:179--210,
  1968.

\bibitem{pietsch1986eigenvalues}
Albrecht Pietsch.
\newblock {\em Eigenvalues and s-numbers}.
\newblock Cambridge University Press, 1986.

\bibitem{rasmussen2009dichotomy}
Martin Rasmussen.
\newblock Dichotomy spectra and {M}orse decompositions of linear nonautonomous
  differential equations.
\newblock {\em Journal of Differential Equations}, 246(6):2242--2263, 2009.

\bibitem{rasmussen2010alternative}
Martin Rasmussen.
\newblock An alternative approach to {S}acker--{S}ell spectral theory.
\newblock {\em Journal of Difference Equations and Applications},
  16(2-3):227--242, 2010.

\bibitem{ruelle1982characteristic}
David Ruelle.
\newblock Characteristic exponents and invariant manifolds in {H}ilbert space.
\newblock {\em Annals of Mathematics}, pages 243--290, 1982.

\bibitem{rybakowski2012homotopy}
Krzysztof Rybakowski.
\newblock {\em The homotopy index and partial differential equations}.
\newblock Springer Science \& Business Media, 2012.

\bibitem{sacker1978existence}
Robert~J Sacker.
\newblock Existence of dichotomies and invariant splittings for linear
  differential systems, {IV}.
\newblock {\em Journal of Differential Equations}, 27(1):106--137, 1978.

\bibitem{sacker1974existence}
Robert~J Sacker and George~R Sell.
\newblock Existence of dichotomies and invariant splittings for linear
  differential systems, {I}.
\newblock {\em Journal of Differential Equations}, 15(3):429--458, 1974.

\bibitem{sacker1976existenceII}
Robert~J Sacker and George~R Sell.
\newblock Existence of dichotomies and invariant splittings for linear
  differential systems, {II}.
\newblock {\em Journal of Differential Equations}, 22(2):478--496, 1976.

\bibitem{sacker1976existenceIII}
Robert~J Sacker and George~R Sell.
\newblock Existence of dichotomies and invariant splittings for linear
  differential systems, {III}.
\newblock {\em Journal of Differential Equations}, 22(2):497--522, 1976.

\bibitem{sacker1978spectral}
Robert~J Sacker and George~R Sell.
\newblock A spectral theory for linear differential systems.
\newblock {\em Journal of Differential Equations}, 27(3):320--358, 1978.

\bibitem{sacker1994dichotomies}
Robert~J Sacker and George~R Sell.
\newblock Dichotomies for linear evolutionary equations in {B}anach spaces.
\newblock {\em Journal of Differential Equations}, 113(1):17--67, 1994.

\bibitem{salamon1988flows}
Dietmar Salamon and Eduard Zehnder.
\newblock Flows on vector bundles and hyperbolic sets.
\newblock {\em Transactions of the American Mathematical Society},
  306(2):623--649, 1988.

\bibitem{selgrade1975isolated}
James~F Selgrade.
\newblock Isolated invariant sets for flows on vector bundles.
\newblock {\em Transactions of the American Mathematical Society},
  203:359--390, 1975.

\bibitem{sell2013dynamics}
George~R Sell and Yuncheng You.
\newblock {\em Dynamics of evolutionary equations}, volume 143.
\newblock Springer Science \& Business Media, 2013.

\bibitem{shen1998almost}
Wenxian Shen and Yingfei Yi.
\newblock {\em Almost automorphic and almost periodic dynamics in skew-product
  semiflows}.
\newblock Number 647 in Memoirs of the American Mathematical Society. American
  Mathematical Soc., 1998.

\bibitem{shvydkoy2006cocycles}
Roman Shvydkoy.
\newblock {C}ocycles and {M}a{\~n}e sequences with an application to ideal
  fluids.
\newblock {\em Journal of Differential Equations}, 229(1):49--62, 2006.

\bibitem{temam2012infinite}
Roger Temam.
\newblock {\em Infinite-dimensional dynamical systems in mechanics and
  physics}, volume~68.
\newblock Springer Science \& Business Media, 2012.

\bibitem{thieullen1987fibres}
Philippe Thieullen.
\newblock Fibr{\'e}s dynamiques asymptotiquement compacts exposants de
  {L}yapounov. {E}ntropie. {D}imension.
\newblock In {\em Annales de l'IHP Analyse non lin{\'e}aire}, volume~4, pages
  49--97, 1987.

\bibitem{wojtaszczyk1996banach}
Przemyslaw Wojtaszczyk.
\newblock {\em {B}anach spaces for analysts}, volume~25 of {\em Cambridge
  Studies in Advanced Mathematics}.
\newblock Cambridge University Press, 1996.

\end{thebibliography}
\bibliographystyle{plain}
%\nocite{*}

\end{document}